\documentclass[10pt]{amsart}
\usepackage{amssymb,url,upref,verbatim, hyperref, amsmath}
\usepackage{amscd}
\usepackage{mathtools}
\usepackage[dvipsnames]{xcolor}
\usepackage[T1]{fontenc}
\usepackage{hyperref,amssymb,amsmath,url,upref,verbatim,xspace,mathrsfs,bbm,enumitem}
\usepackage{cancel}
\usepackage{cleveref}
\newcommand{\RED}[1]{#1}
\newcommand{\blue}[1]{#1}
\theoremstyle{plain}
\newtheorem{theorem}{Theorem}[section]
\newtheorem{Proposition}[theorem]{Proposition}
\newtheorem{lemma}[theorem]{Lemma}
\newtheorem{corollary}[theorem]{Corollary}

\theoremstyle{definition}
\newtheorem{definition}[theorem]{Definition}

\newtheorem{example}[theorem]{Example}

\newtheorem{remark}[theorem]{Remark}

\newtheorem{MRF}[theorem]{Main Recursive Formula}
\newtheorem*{claim*}{Claim}

\DeclareMathOperator{\Double}{R}
\DeclareMathOperator{\Alt}{A}

\DeclareMathOperator{\vect}{vec}
\DeclareMathOperator{\nvect}{n-vec}

\DeclareMathOperator{\bin}{\bf{b}}
\newcommand{\RomanNumeralCaps}[1]
{\MakeUppercase{\romannumeral #1}}

\newcommand{\didem}[1]{I^{\Delta}_{#1}}
\newcommand{\dnil}[1]{N^{\Delta}_{#1}}

\newcommand{\per}[1]{P_{#1}}

\newcommand{\lcm}{\mathrm{lcm}}
\newcommand{\Z}{\mathbb{Z}}
\newcommand{\N}{\mathbb{N}}

\newcommand{\trace}{\mathrm{tr}}

\newcommand{\peq}{\equiv_{\nu}}

\newcommand{\lead}{e_{\gamma}}

\begin{document}
		\date{\today}
	\author{Luisa Fiorot}
	\author{Riccardo Gilblas}
	\author{Alberto Tonolo}

\title[]{Modular binomials with an application to periodic sequences}

\thanks{}

\address{Luisa Fiorot\\ Dipartimento di Matematica ``Tullio Levi-Civita''\\
	Via Trieste, 63
	35121 Padova Italy\\ \texttt{luisa.fiorot@unipd.it}}
\address{	
	Centro interdipartimentale di ricerca "I–APPROVE – International Auditory Processing Project in Venice"
c/o Dipartimento di Neuroscienze - DNS
via Belzoni 160 - 35121 - Padova Italy
	}

\address{Riccardo Gilblas\\ Dipartimento di Matematica ``Tullio Levi-Civita''\\
	Via Trieste, 63
	35121 Padova Italy\\ \texttt{riccardo.gilblas@phd.unipd.it}}

\address{Alberto Tonolo\\  Dipartimento di Matematica ``Tullio Levi-Civita''\\
	Via Trieste, 63
	35121 Padova Italy\\ \texttt{alberto.tonolo@unipd.it}}
	\address{	
	Centro interdipartimentale di ricerca "I–APPROVE – International Auditory Processing Project in Venice"
c/o Dipartimento di Neuroscienze - DNS
via Belzoni 160 - 35121 - Padova Italy
	}

\keywords{Periodic Sequences, Modular Binomial Coefficients, difference operator, sum operator}

\subjclass[2020]{11B50, 11B65}

\maketitle

\begin{abstract}
We study, through new recurrence relations for certain binomial coefficients modulo a power of a prime, the evolution of the primitives of a modular periodic sequence. We prove that we can reduce to study primitives of constant sequences and that the latter are controlled by modular binomial coefficients.  Finally we apply our results to describe the dynamics of the primitives of the sequence considered by the Romanian composer Vieru in his \emph{Book of Modes}\cite{V}.

 \end{abstract}

In this paper, we investigate periodic sequences in modular arithmetic and their transformations when the difference operator $\Delta$ and the sum operator $\Sigma$ are applied to them. The operator $\Delta$ allows us to examine the differences between consecutive terms in a sequence, while the operator $\Sigma$ serves as a complementary tool, enabling us to study the cumulative sums of sequences and understand their overall behavior. Similar to conventional integration, the sum operator is determined up to a constant term. We refer to the obtained sequences by applying the operator $\Sigma$ a finite number of times as primitives of a periodic sequence. The study of the period of modular sequences has been and continues to be of interests in combinatorics and its applications to computer science and cryptography \cite{Nat,ZLG}.

Any periodic sequence can be uniquely decomposed into the sum of an idempotent sequence and a nilpotent sequence. A sequence is considered idempotent when applying the difference operator $\Delta$ multiple times results in the same sequence. Conversely, nilpotent sequences vanish after a certain number of applications of the difference operator $\Delta$. By studying the idempotent and nilpotent parts separately, we gain a deeper understanding of the dynamics and behavior of the original sequence. 

In the context of periodic sequences taking values in $\mathbb Z/m\mathbb Z$ with $m\in\mathbb N$, we focus on studying the evolution of their periods and the $p$-adic evaluation of their elements when the sum operator $\Sigma$ is applied. An important result is that any nilpotent sequence $f$ can be expressed as a finite sum of primitives of constant sequences (refer to \Cref{lemma:nilpotentwithconstants}), and the period of the primitives of $f$ is ultimately determined by one of these constant sequences (refer to \Cref{theorem:leading_term}). Similarly, the primitives of an idempotent sequence $f$ can also be described in terms of primitives of constant sequences: they are the sum of an element from the $\Delta$-orbit of the idempotent sequence $f$ and a sum of primitives of constant sequences. Once again, the period of the primitives of $f$ is determined by one of these constant sequences (refer to \Cref{idem_constants}).

Thus, the primitives of constant sequences play a fundamental role in our analysis. By  \Cref{lemma:binomialconstants} the $n$-th entry of the $s$-primitive of a constant $c$ is equal to $c$ times the binomial coefficient $\binom ns$ in $\mathbb Z/m\mathbb Z$. Consequently, the study of modular binomial coefficients becomes significant. 
Many of the great mathematicians of the nineteenth century considered problems involving binomial coefficients modulo a prime power (for instance Babbage, Cauchy, Cayley, Gauss, Hensel, Hermite, Kummer, Legendre, Lucas and Stickelberger). Several attempts of generalising these classical results can be found in \cite{AHP, DW90, DW,F47, G, H71, H73,  KW, LT, Mat}. Our paper aims to provide a new contribution in this research area.

By applying the Chinese Remainder Theorem, we can reduce the study of binomial coefficients to $\mathbb Z/p^\ell\mathbb Z$, where $p$ is a prime number and $\ell\geq 1$. We derive recurrence relations for certain binomial coefficients, enabling efficient computation of their $p$-adic evaluation. This occurs precisely when the denominator exhibits patterns of the following types: $\underbracket[0.1ex]{p-1, \dots,p-1}_{\ell}$, $\underbracket[0.1ex]{0, \dots,0}_{\ell}$, and $\underbracket[0.1ex]{p-1, 0, \dots,0}_{\ell}$ (refer to \Cref{seq:11,seq:000,seq:101}). 

As an application, we provide a comprehensive answer to two questions posed by the Romanian composer Anatol Vieru (1926-1998). In the context of 1960s musical serialism, Vieru in his  \textit{Book of Modes}\cite{V} explores a composition technique based on periodic sequences with values in $\Z_{12}$. If a sequence $f$ in $\mathbb Z_{12}$ represents the pitch classes of an initial musical theme, Vieru decodes a musical aspect (such as rhythm, harmony, tone color, or dynamics) from the primitives of $f$ using a suitable dictionary. Employing this technique, Vieru composes several pieces, including Symphony No. 2 and ``Zone d'Oublie''. Manipulating the Messiaen's second mode of limited transposition, he got the sequence $(2,1,2,4,8,1,8,4)$. 
The questions were to provide a formula for the period of its primitives, and, additionally, to explain why the values 4 and 8 proliferate in the coefficients of its primitives. These questions have been formalised in precise mathematical terms in the papers \cite{AV, AAV2}. Our complete solutions to these questions have been announced without proofs on the paper \cite{FGT}.

We believe that the techniques and ideas we have developed in our application to address Vieru's questions can be applicable in various other scenarios involving the study of primitives of periodic sequences.

In \Cref{section:periodic_sequences}, we define periodic sequences with values on modular integers. We introduce the difference operator $\Delta$ and the sum operator $\Sigma$ and we show the fundamental link between the primitives of constant sequences and the modular binomial coefficients.

In \Cref{section:decomposition}, we present some results about decomposing the $\Z_{m}$-module $P_{m}$ of periodic sequences: the decomposition in nilpotent and idempotent part, and  the decomposition in primes through the Chinese Remainder Theorm.

In \Cref{section:period}, we clarify the link between primitives of constant sequences and modular binomial coefficients. Then we present several results showing how it is always possible to reduce the study of primitives of generic sequences to primitives of constant sequences.

In \Cref{section:recursive_formula}, we provide a recursive formula for binomial coefficients modulo the power of a prime integer $p$.
In several cases, this formula allows to reduce the complexity of the computation of the $p$-adic evaluation of sequences of binomial coefficients.

As an example of actual applicability of our results, in \Cref{section:Vieru_sequence} we apply the recursive formulas to the peculiar sequence that arises from
the mathematical-musical problem posed by
Vieru in  \cite{V}. This allows us to express the coefficients of the primitives of such a sequence in a recursive way.

\section{Periodic sequences in $ \Z_m $ and the operators $ \Delta $ and $ \Sigma $ }
\label{section:periodic_sequences}

In this section we introduce the periodic sequences over $ \Z_m = \Z/m\Z$ with $ m \ge 2 $, the difference operator $ \Delta $ and the sum operator $ \Sigma $.

\subsection*{Periodic sequences in $ \Z_m $}
Let $m\geq 2$ be a natural number. We denote by $S_m:=\mathbb Z_m^{\mathbb N}$ the $\mathbb Z_m$-module of all sequences with values in $\mathbb Z_m$. 

\noindent The \emph{shifting} operator $ \theta $ is the endomorphism of   $ S_m $ acting on $ f \in S_m $ as:
\[ \theta(f)(n) := f(n+1) \quad \forall n \in \mathbb N.\] \noindent A sequence $ f \in S_m $ is said \emph{periodic} if there exists $ j \ge 1 $ such that $ \theta^j (f) = f $, i.e. $ f \in \ker(\theta^j - \text{id}) $. We denote by $P_m$ the $ \mathbb{Z}_m $-submodule of periodic sequences in $ S_m $: \[ P_m := \bigcup_{j \ge 1} \ker(\theta^j - \text{id}). \] Given a periodic sequence $ f \in P_m $, we say that it has \emph{period} $ \tau(f) $ if $ \tau(f) $ is the minimum positive integer such that $ \theta^{\tau(f)} f= f $. Furthermore $ \theta^k f = f $ if and only if $ \tau(f) \mid k.  $ Since $\tau(f)=\tau(\theta(f))$, $ \theta$ restricts to an endomorphism of  $P_m $.

Let  \(f\) be a sequence of period \(\tau \). Since it is determined by its values \( f(0)\), \dots, \(f(\tau-1)\), we will write
\[ f = (f(0), f(1), \dots, f(\tau), f(\tau+1), \dots )=: [f(0), f(1), \dots, f(\tau-1)].\]
\noindent In particular, \RED{for any $ c\in\mathbb Z_m $,} $[c]$ denotes a constant sequence.

We define the trace of $ f $ to be:

\[ \trace f := \sum_{i=0}^{\tau - 1} f(i). \]

\subsection*{The operator $ \Delta $}

We define on $ S_m $ the difference operator: $  \Delta := \theta - \text{id}$.
It restricts to an operator of $P_m$  since the period of \(\Delta f\) divides the period of \(f\).

 We say that a periodic sequence $ f \in P_m $ is \emph{nilpotent} (resp. \emph{idempotent}) if there exists $ \eta \ge 1 $ such that $ \Delta^\eta f = 0$ (resp.  $ \Delta^\eta f = f$). The minimal $ \eta $ satisfying this condition is said to be the \emph{nilpotency} (resp. \emph{idempotency}) \emph{index} of $ f $.
We denote by $ \didem{m} $ the submodule of $ \per{m} $ of idempotent sequences and by $ \dnil{m} $ the submodule of nilpotent sequences. 

\begin{example}
	The sequence $ f = [0 , 1 , 2 , 3] \in \per{4} $ is nilpotent of index 2, while the sequence $ g = [2 , 1] \in \per{3}$ is idempotent of index 1.
\end{example}

\subsection*{The operator $ \Sigma $}

Given $ c \in \mathbb{Z}_m $ we define the sum function $ \Sigma_c $ as follows: for every sequence $ f \in S_m $,
\begin{equation}\label{eq:sigma}
	\Sigma_c f (n) := \begin{cases}
		c  \text{ if }  n = 0 \\
		f(n-1) + \Sigma_c f(n-1) \text{ if } n \ge 1.
	\end{cases} 
  \end{equation}
For every $ c  \in \mathbb{Z}_m $, $ \Sigma_c $ acts as right inverse for $ \Delta $, i.e. $  \Delta \circ \Sigma_c = \text{id} $.
Also it is a matter of explicit computation to find that
\[ (\Sigma_{f(0)} \circ \Delta ) (f) = f\]
while with a generic $ c $ one gets:
\[ (\Sigma_c \circ \Delta ) (f) = f  + [c - f(0)].\]
Notice that $ \Sigma_c f = \Sigma_0 f + [c] $ for every $  c \in \mathbb{Z}_m  $ and $ f \in S_m $, so
\(\Sigma_{c}\) defines an endomorphism of
\(S_{m}\) only when \(c = 0\).
We will write $ \Sigma $ instead of $ \Sigma_0 $  and we will \RED{often} use the notation $f^{s} := \Sigma^{s} f$, in order to keep the notation clean. We call \RED{\emph{primitives}  of $ f $ the sequences  $ f^s $, $s\geq 1$}.


\noindent Notice that \(\Sigma \) gives in fact an endomorphism of \(\per{m}\). If \(f \in \per{m}\) has period \(\tau\), then \( \theta^{\tau m}(\Sigma f)= \Sigma f\).
Precisely we have

\begin{lemma} \label{lemma:primitiveperiod}
For any $ f \in \per{m} $, if $ h $ is the additive order of $ \trace f $ in $ \Z_m $, then
  \[ \tau(\Sigma f ) = h \cdot \tau(f).\]
\begin{proof}
   We already observed that
   \( \tau(\Delta g)   \mid  \tau(g) \),
   so for \(g = \Sigma f\) we have
   \[ \tau(f) = \tau(\Delta (\Sigma f)) \mid \tau(\Sigma f). \]
   Since by the Fundamental Theorem of finite calculus \cite[Th. 6.27]{MT2} 
   \[\Sigma f(\tau m+i) -   \Sigma f (i) = \sum_{j=i}^{\tau m + i - 1}f(j)  =  m \sum_{j=i}^{\tau + i - 1}f(j) = m \cdot \trace f ,\]
we deduce that \(\tau(\Sigma f) = h \cdot\tau(f)\).
\end{proof}
\end{lemma}

\begin{example}
	If $ f $ is the constant sequence $ [c] $, one has:
	\[ \Sigma f (n+1) = \Sigma f (n+1) - \Sigma f (0) = \sum_{j=0}^n c =(n+1) c. \]
	Hence the sequence $ \Sigma f $ has period equal to the additive order of $ c $ in $ \Z_m $.
\end{example}

\section{Decomposition of $ \per{m} $}
\label{section:decomposition}
In this section we introduce in cascade two decompositions for the $\mathbb Z_m$-module  $ \per{m} $. The first one decomposes $ \per{m} $ as the direct sum of $ \didem{m} $ and $ \dnil{m} $, \RED{the submodules of idempotent and nilpotent sequences respectively.}
The second one is the standard decomposition into primes, using the factorization of $ m $. 
\subsection*{Decomposition in idempotent and nilpotent part}
Let $ f \in \per{m} $ be a sequence of period $ \tau $ and consider the set $ A = \{ \Delta^i f \mid i \in \N\} $. $ A $ is a subset of the set of sequences having the period dividing $ \tau $, hence $ A $ is finite. So take the minimal $ M \in \N $ such that there exists $ u < M $ satisfying
\[ \Delta^M f = \Delta^u f. \]
If $ t := M - u $ then $ \Delta^{t+u} f = \Delta^u f  $ and \( \Delta^{u}f, \Delta^{u+1}f, \dots, \Delta^{M-1}f \)
are distinct sequences. Define $ \bar{k} $ to be the minimal $ k \in \N $ such that $ kt \ge u $.
It is \(u \le \bar{k}t < M \).
Denote:
\begin{equation}
	f_I := \Delta^{\bar{k} t} f \qquad f_N := f - f_I.
\end{equation}

\begin{lemma}\label{lemma:antifitting}
  With the above notation, $ f = f_I + f_N $ is the unique decomposition of $ f $ as a sum of an idempotent and a nilpotent sequence.
  The sequence  $ f_N  $  (resp.  $ f_I $) has nilpotency (resp. idempotency) index $ \bar{k} t $ (resp. $ t $). Moreover \(\tau = \mathrm{lcm}\{\tau(f_{I}), \tau(f_{N})\}\).
\end{lemma}
 
\begin{proof}
	The sequence $ f_I $ is idempotent since \[  \Delta^t f_I = \Delta^t (\Delta^{\bar{k}t} f) =  \Delta^{\bar{k}t -u } (\Delta^{u+t} f) = \Delta^{\bar{k}t -u } (\Delta^u f ) = \Delta^{\bar{k}t } f = f_I. \]
	The minimality of $ t $ comes from the fact that $ \{ \Delta^{\bar{k}t+i} f \mid 0 \le i \le t-1\} $ has cardinality $ t $.
	The sequence $ f_N $ is nilpotent since
	\[ \Delta^{\bar{k}t} f_N = \Delta^{\bar{k}t} (f - f_I) = \Delta^{\bar{k}t} f - \Delta^{\bar{k}t} f_{I} = f_{I} -f_{I} = 0. \]
	The minimality of $ \bar{k}t $ follows from the minimality of $ \bar{k} $.
	
	\noindent This decomposition is unique: by contradiction take $ f = f_I' + f_N' $. One has that $ f_I - f_I' = f_N' - f_N $ is both nilpotent and idempotent thus it is equal to $ 0 $.\\
	Furthermore, one clearly has \( \tau \mid \mathrm{lcm}\{\tau(f_{I}), \tau(f_{N})\}\). Since \(\Delta\) and \(\theta\) commute, \(\theta^{\tau}(f_{N})\) (resp. \(\theta^{\tau}(f_{I})\)) is nilpotent (resp. idempotent).
	From \[ f = \theta^{\tau} (f) = \theta^{\tau}(f_{N})+  \theta^{\tau}(f_{I})\] and the uniqueness of the decomposition, one gets \(\theta^{\tau}(f_{N}) = f_{N}\) and  \(\theta^{\tau}(f_{I}) = f_{I}\), thus
	\( \tau(f_{N}),\tau(f_{I}) \mid \tau \). Hence  \( \tau = \mathrm{lcm}\{\tau(f_{I}), \tau(f_{N})\}\).
\end{proof}

\subsection*{Decomposition with primes}

Given the prime factorization $  m = \prod_{i= 1}^t p_i^{\ell_i} $,  
the group isomorphism $ \Z_m \to \bigoplus_{i =1}^t \Z_{p_i^{\ell_i}} $ 
gives rise to an isomorphism of $ \Z_m $-modules
\begin{align*}
	\per{m} \longrightarrow  & \bigoplus_{i = 1}^t \per{p_i^{\ell_i}} \\
	f \longmapsto  & (f_{p_i})_{1 \le i \le t}
\end{align*}
where $ f_{p_i}  (n) \equiv f(n) \mod p_i^{\ell_i} $. The sequence \(f_{p_{i}}\) is the  $ p_i $-part of $ f $.  The inverse of this morphism is given by the Chinese Remainder Theorem.

\blue{As a consequence, one can easily prove the following lemma.}
\begin{lemma} \label{nilpotencyislocal} \blue{ \cite[Prop. 13 and Prop. 16]{AV} }
  A sequence $ f \in \per{m} $ is nilpotent (resp. idempotent) if and only if  the $ p_i $-part $  f_{p_i}$ is nilpotent (resp. idempotent) for every $ i $.
  The nilpotency (resp. idempotency) index $ \eta $ coincides with the maximum  (resp. least common multiple) of the nilpotency (resp. idempotency) indices $ \eta_i $ of $  f_{p_i}$ for $ i  = 1, \dots, t. $ Moreover, the period of $ f $ satisfies:
	\[ \tau(f) = \mathrm{lcm} \{\tau(f_{p_i})\}_{1 \le i \le t} .\]
\end{lemma}


The primes decomposition, \Cref{lemma:antifitting,nilpotencyislocal}  imply the following isomorphisms: \[ \per{m} = \bigoplus_{i = 1}^t \left( \didem{p_i^{\ell_i}} \oplus \dnil{p_i^{\ell_i}} \right)  \qquad \didem{m} = \bigoplus_{i = 1}^t \didem{p_i^{\ell_i}}  \qquad \dnil{m} = \bigoplus_{i = 1}^t \dnil{p_i^{\ell_i}} .\]
Thus we can always reduce to study sequences on $ \Z_{p^\ell} $.

\begin{theorem} \label{theorem:nilpotent_period} \cite[Th. 7]{AV}
  Let $ f \in \per{p^\ell} $ be a periodic sequence.
  Then  $ f \in \dnil{p^\ell} $  if and only if $ \tau(f) = p^t $ for $ t \in \N $.
\end{theorem}

\begin{remark}
The period of an idempotent sequence $ f \in \didem{p^\ell} $ may or may not be divisible by $p$. 
Using a generic computer algebra system one can easily check that  the sequence $[1,1,1,0,0,2,0,0,0,2,2,2,0,0,1,0,0,0] \in \per{3} $ is idempotent (of index 9) and it has period 18, and  the sequence $[0,2,0,0,1] \in \per{3} $ is idempotent (of index 80) and it has period 5.
\end{remark}

\begin{definition}
	Consider $ f \in \per{p^\ell} $ of period $\tau =  q p^t $ with $ p \nmid q $. The \emph{$ p^t $-periodised sequence} of $ f $ is the sequence:
	\[  \sum_{j =1}^{q} \theta^{jp^t} f = \theta^{p^t} f + \theta^{2p^t} f + \dots + \theta^{\tau - p^t} f + f. \]
	It is easy to verify that it has period dividing $ p^t $ \RED{and hence it is nilpotent}.
\end{definition}

\begin{Proposition}\label{remark:idemiffpperzero} \cite[Th. 17]{AV}
Given $ f \in \per{p^\ell} $ of period $ \tau = q p^t $ with $ p \nmid q $, the  nilpotent part $ f_N $ of $ f $  coincides with the $ p^t $-periodised sequence of $ f $ multiplied by $ q^{-1} \mod p^\ell $.
\end{Proposition}

\begin{corollary}\label{cor:idempotent}
	Let $ f \in \per{p^\ell} $ be a periodic sequence.
	\begin{enumerate}
		\item If $ f $ is idempotent, then $ \trace f = 0 $.
		\item If  $ \trace f = 0 $ and $p\nmid\tau(f)$ then $f$ is idempotent.
%
	\end{enumerate} 
	\begin{proof}
1. If $f$ is idempotent with idepotency index $\eta$, from $\Delta^{\eta}f=f$ one gets
$\Delta^{\eta-1} f-[\Delta^{\eta-1}f(0)]=\Sigma f$. By the idempotency of $f$ and 	\Cref{lemma:primitiveperiod} one has
\[\tau(f)=\tau\left(\Delta^{\eta-1} f\right)=\tau\left(\Delta^{\eta-1} f-[\Delta^{\eta-1}f(0)]\right)=\tau\left(\Sigma f\right)=h\cdot\tau(f)\]
where $h$ is the additive order of $\trace f$. Then $h=1$ and hence $ \trace f=0 $.\\
2. If $ p \nmid \tau $, the $ p^t $-periodised of $ f $ coincides with the constant sequence $ [\trace f] =[0]$. By \Cref{remark:idemiffpperzero}, $f$ is idempotent.
	\end{proof}
\end{corollary}

Given a prime $p$ and a natural number $m$, we denote by $\nu_p(m)$ (or simply by $\nu(m) $ when the prime $p$ is clear in the context) the $p$-adic evaluation of $m$, i.e., the highest
power of $p$ dividing $m$. In particular the $p$-adic evaluation of 0 is infinite. The elements in any non zero coset in $\mathbb Z_{p^\ell}$ have the same $p$-adic evaluation: therefore setting $\nu_p(0+p^\ell\mathbb Z)=\infty$ the $p$-adic evaluation can be defined also on $\mathbb Z_{p^\ell}$.
\begin{definition}\label{def:genvec}
Let $f\in \didem{p^\ell}\cup\dnil{p^\ell}$ be a either idempotent or nilpotent periodic sequence  with idempotency or nilpotency index $\eta$. We call \emph{generating vector} of $f$ the ordered $\eta$-tuple
\[\vect(f)=(e^{f}_0 , e^{f}_1 , ..., e^{f}_{\eta-1})\in\mathbb Z_{p^\ell}^\eta\quad e^{f}_i = \Delta^i {f} (0), \ 0\leq i<\eta.\]
The last entry of $\vect(f)$ with minimal $p$-adic evaluation is the \emph{leading component} of $f$.
\end{definition}

\begin{example}\label{example:vieru}
\RED{Consider the sequence $ V= [2,1,2,4,8,1,8,4] \in \per{12} $.
The $ 2 $- and the $ 3 $-parts of $ V $ are
\[ V_2 = [2,1,2,0,0,1,0,0] \in \per{4} \qquad V_3 = [2,1] \in \per{3} .\]}
The sequence $ V_3 $ has period $ 2 $ and   $ \trace (V_3) = 0 $: hence  it is idempotent by Corollary~\ref{cor:idempotent}. Clearly it has idempotency index 1; then $\vect(V_3)=(2)$.
The sequence $ V_2 $ has period $ 8 $ and hence by Theorem~\ref{theorem:nilpotent_period}  it is nilpotent. Since
	$ \Delta^5 V_2 = 0 $
	while $ \Delta^4 V_2 = [2] $, the sequence $ V_2 $ has nilpotency index $ 5 $; then $\vect(V_2)=(2,1,2,0,0)$.
By the Chinese Remainder Theorem, the sequences $ V_2 $ and $ V_3 $ correspond respectively  to the following sequences in $\per{12} $:
	\[ \tilde{V}_2 = [6,9,6,0,0,9,0,0] \qquad \tilde{V}_3 = [8,4] \]
	and $ V = \tilde{V}_2 + \tilde{V}_3 $.
	\end{example}

\begin{Proposition}\label{vectno0}
Any idempotent (resp. nilpotent) sequence $f\in\didem{p^\ell}\cup\dnil{p^\ell}$ determines univocally its generating  vector $\vect(f)$.
\end{Proposition}
\begin{proof}
It follows immediately from \Cref{def:genvec} that $\vect$ restricted to $\didem{p^\ell}$ (resp. $\dnil{p^\ell}$) is an homomorpism of $\mathbb Z_m$-modules. Then it is sufficient to prove that $\vect(f)=0$ implies $f=0$.\\
If $f$ is idempotent with idempotency index $\eta$ and $\vect(f)=0$, then
\begin{align*}
0&=f(0)\\
0&=\Delta f(0)=f(1)\\
&\ \,\vdots\\
0&=\Delta^{\eta-1}f(0)=\Delta^{\eta-2}f(1)=\cdots f(\eta-1)\\
0&=f(0)=\Delta^{\eta} f(0)=\Delta^{\eta-1}f(1)=\cdots=f(\eta)\\
&\ \,\vdots
\end{align*}
and hence $f=0$.\\
If $f\not=0$ is nilpotent with nilpotency index $\eta$, necessarily $\geq 1$, then $\Delta^{\eta-1}f$ is a non zero constant sequence and hence the $(\eta-1)$-th component $\Delta^{\eta-1}f(0)$ of the generating vector is non zero.
\end{proof}
%


\begin{Proposition}\label{lemma:nilpotentwithconstants}
Let $f\in \dnil{p^\ell}$ be a nilpotent sequence of nilpotency index $\eta$ with generating vector $\vect(f)=(e^f_0,...,e^f_{\eta-1})$. Then 
$\displaystyle f = \sum_{i=0}^{\eta-1} \Sigma^i [e^f_i]$.
\begin{proof}
We proceed by induction on $ \eta: $
\begin{itemize}
	\item $ \eta = 1 $ means $ \Delta f = 0 $, so $ f = [c]  = [f(0)]$.
			
\item Suppose that the statement holds for $ \eta=t $. If $ f $ has nilpotency index $ \eta=t+1 $, then $ \Delta f $ has nilpotency index $ t $ and by inductive hypothesis:
\[ \Delta f = \sum_{i=0}^{\eta-1} [e^{\Delta f}_i]^{i} = \sum_{i=0}^{\eta-1} [\Delta^i (\Delta f)(0)]^{i}  =  \sum_{i=0}^{\eta-1} [\Delta^{i+1}  f(0)]^{i} = \sum_{i=0}^{\eta-1}  [e^f_{i+1}]^{i} .\]
Since
$f = \Sigma_{f(0)} \Delta f =[f(0)] + \Sigma \Delta f$
we obtain that
\[	f =    [f(0)] + \Sigma \left(\sum_{i=0}^{\eta-1} [e^f_{i+1}]^{i} \right)
	=    [e^f_0] + \sum_{i=1}^\eta [e^f_i]^{i}
	=   \sum_{i=0}^\eta [e^f_i]^{i}.
\]
\end{itemize}
\end{proof}

\end{Proposition}


\RED{\begin{remark}
  With the notation of the previous proposition, notice that
  \[\tau(f) \mid \max\{\tau(\Sigma^{i} [e^f_i])\,:\, 0\leq i<\eta\}.\]
  In general we have not the equality.
  For example \(f = [0,2]\in \dnil{8} \) is a nilpotent sequence of nilpotency index $\eta=3$. It has period 2 and $\nvect(f)=(0,2,4)$.
  By \Cref{lemma:nilpotentwithconstants} one has  \(f = \Sigma[2] + \Sigma^{2}[4]\), and both \(\Sigma[2]\) and  \(\Sigma^{2}[4]\) have period 4.
\end{remark}
}
\section{Constant sequences and their primitives}\label{section:period}

\blue{
	The first lemma we are presenting in this section displays the connection between the primitives of constant sequences and the binomial coefficients.
Then, we will present several results that show how it is possible to reduce the study of primitives of a generic sequence to primitives of constant sequences.
	}

\begin{lemma}  \label{lemma:binomialconstants}
	If $ [c] $ is a constant sequence in $ \per{m} $, then
	\[[c]^s(n)= \Sigma^{s}[c](n)\equiv_m c \binom ns\quad\forall s,n\geq 0. \]
	\begin{proof}
		Since $ \Sigma $ is linear, it is enough to prove the statement for  $ c = 1 $. We proceed by induction on $ s $: the case $ s=0 $ is immediate. Now we suppose true the statement for $ s $ and we prove it for $ s + 1 $. We proceed by induction on $ n $. For $ n = 0 $, we have:
			\[ [1]^{s+1}(0)= \Sigma^{s+1} [1] (0) = 0 = \binom{0}{s+1}. \]
			Now suppose that $ [1]^{s+1} (n) = \binom{n}{s+1} $; we prove the statement for $ n+1: $
			\begin{align*}
				[1]^{s+1} (n+1) = [1]^s (n) + [1]^{s+1} (n) = \binom{n}{s} + \binom{n}{s+1} = \binom{n+1}{s+1}.
			\end{align*}
	\end{proof}
\end{lemma}


\blue{Hence we are reduced to study binomial coefficients modulo $ p^\ell $.
In order to do so, the main tool is \emph{Kummer's Theorem} \cite{K52}.}
This result  says that, given a prime $p$, for given integers $n\geq m\geq0$ with $p$-adic representation
\[n=\lfloor a_sa_{s-1}...a_1a_0 \rfloor_p,\quad s = \lfloor b_sb_{s-1}...b_1b_0 \rfloor_p,\]  the $p$-adic valuation $\nu_p\left(\binom ns \right)$ of the binomial coefficient $n$ over $s$ is equal to the number of borrows in the subtraction $\lfloor a_sa_{s-1}...a_1a_0\rfloor_p-\lfloor b_sb_{s-1}...b_1b_0\rfloor_p$.

\begin{example}
	Consider the numbers
	\[798= \lfloor 1002120 \rfloor_3, \quad 454=\lfloor 121211 \rfloor_3.\]
	Let us compute the $3$-adic valuation of $\binom{798}{454}$:
		\[\begin{tabular}{ cccccccl} 
			$ \check{1}$ & $\check{0}$ & 0&$\check{2}$&1&$\check{2}$&0&- \\ 
			&1&2&1&2&1&1&= \\ 
			\hline\\
			&1&1&0&2&0&2&
		\end{tabular}\qquad\Longrightarrow\quad \nu_3\left(\binom{798}{454}\right)=4.\]
\end{example}

Kummer's Theorem allows to characterize the period of primitives of constant sequences. First let us prove a technical lemma:

\begin{lemma}\label{lemma:technical}
	Let $ [c] $ be a non zero constant sequence in $ \per{p^\ell} $ with $ \nu_{p}(c)=t. $ Consider \(s \in \N\) and  \(s =  \lfloor a_ka_{k-1}\cdots a_1a_0 \rfloor_p \)
	with $a_k\not=0$.
The following equality holds: 
\begin{align*}
	\nu_p\left(\sum_{n=0}^{p^{\ell-t+k}-1} [c]^s(n) \right) =\begin{cases}
	\ell-1 & \text{ if } a_k=a_{k-1}=\cdots=a_0=p-1\\ 
	\geq \ell &\text{ otherwise.} 
\end{cases}
\end{align*}

\begin{proof}
From Lemma~\ref{lemma:binomialconstants} and a well known binomial identity (e.g.,\cite[Prop. 2.56]{MT1}) we have:
\[ \sum_{n=0}^{p^{\ell-t+k}-1} [1]^s(n) = \sum_{n=0}^{p^{\ell-t+k}-1} c \binom{n}{s} = c \binom{p^{\ell-t+k}}{s+1}.\]
If $ s \neq  \lfloor\underbracket[0.1ex]{(p-1) \cdots (p-1)}_{k+1} \rfloor_p $, we can write $ s+1 = \lfloor b_k\cdots b_0 \rfloor_p $.
Suppose that $ h $ is the smallest index such that $ b_h \neq 0 $.
The computation $ p^{\ell-t+k} - (s+1) $ in base $ p $:
\[
\begin{matrix}
1 \underbracket[0.1ex]{0 \dots 0}_{\ell-t-1} & 0 & 0 & \cdots & 0 & \underbracket[0.1ex]{0 \cdots 0}_h \\
& b_k & b_{k-1} & \cdots & b_h & \underbracket[0.1ex]{0 \cdots 0}_h
\end{matrix}
\]
gives rise to $ \ell-t+k-h $ borrows. Thus by Kummer's Theorem

\[ \nu_p \binom{p^{\ell -t +k}}{s+1} = \ell -t +k -h.
 \]
Since $\nu_p(c)=t$, we conclude 
$\nu_p\left( c \binom{p^{\ell-t+k}}{s+1}\right)=\ell +k -h \geq \ell$.

Now suppose $ s = \lfloor\underbracket[0.1ex]{(p-1) \cdots (p-1)}_{k+1} \rfloor_p $;
in this case $ s+1 = p^{k+1} $. The computation $ p^{\ell-t+k}-p^{k+1} $  in base $ p $:
\[
\begin{matrix}
1 \underbracket[0.1ex]{0 \dots 0}_{\ell -t - 1} & 0 & 0 & \cdots & 0 & \underbracket[0.1ex]{0 \cdots 0}_h \\
{\phantom{1 0 \dots}} 1 & 0 &0 & \cdots & 0 & \underbracket[0.1ex]{0 \cdots 0}_h
\end{matrix}
\]
gives rise to $ \ell-t-1 $ borrows, thus by Kummer's Theorem we get:
\[ \nu_p \left(c \binom{p^{\ell-t+k}}{s+1}\right)=\nu_p  (c)\cdot\nu_p\left(\binom{p^{\ell-t+k}}{p^{k+1}}\right)  = {\ell-1}. \]
 \end{proof}
\end{lemma}

Let us compute the period  of primitives of constant sequences  
(see also \cite{Z}, \cite{SV}, \cite{LT}):
\begin{Proposition}\label{prop:periodconstants}
  Let $[c]$ be a non zero constant sequence in $ \per{p^\ell}$ with $t = \nu_{p}(c)$.
  Let $s\in\mathbb N$ and $s = \lfloor a_{k_{s}}a_{k_{s}-1}\cdots a_1a_0 \rfloor_p $, $a_{k_{s}}\not=0$.
  The sequence $\Sigma^s[c]$ has period $p^{\ell-t+k_{s}}$.
	
\begin{proof}
We proceed by induction on $ s$.
\begin{itemize}
	\item For $ s = 1 $, $ k_s = 0 $, $\tau([c])=1$ and \(\trace([c])=c\), so by Lemma~\ref{lemma:primitiveperiod} the sequence \( \Sigma  [c] \) has period $ p^{\ell-t} $.
			
	\item Suppose that the statement is true for $ s $, we prove it for $ s+1 $:
			
\begin{itemize}
  \item  If $ s \neq  \lfloor(p-1) \cdots (p-1) \rfloor_p $,
  we have $ k_s = k_{s+1} $ and by inductive hypothesis the period of \([c]^{s}=\Sigma^s[c]\) is \(p^{\ell -t + k_{s}}\).
Lemma~\ref{lemma:technical} gives \[ \trace([c]^{s}) = \sum_{n=0}^{p^{\ell-t+k_{s}}-1} [c]^{s}(n)= 0. \]
Thus by Lemma~\ref{lemma:primitiveperiod} one has \( \tau([c]^{s+1}) = \tau([c]^s) = p^{\ell-t+k_s} = p^{\ell-t+k_{s+1}}\).

	\item If $ s=  \lfloor(p-1) \cdots (p-1) \rfloor_p $,
	we have $ k_{s+1} = k_s +1 $ and  Lemma~\ref{lemma:technical} gives 
	$ \nu_p(\trace([c]^s)) = \ell-1 $. By Lemma~\ref{lemma:primitiveperiod}
	\[  \tau([c]^{s+1}) = p \tau([c]^{s}) = p p^{\ell-t+k_s} = p^{\ell-t+k_s +1} = p^{\ell-t+k_{s+1}}.\]
\end{itemize}
\end{itemize}
\end{proof}
\end{Proposition}

\begin{theorem}\label{theorem:leading_term}
\RED{If $\lead$ is the leading component of the generating vector $\vect(f)$ of a nilpotent periodic sequence $ f \in \dnil{p^\ell} $, then
\[\tau(\Sigma^s f) = \tau(\Sigma^{s+\gamma} [\lead]) \quad\text{for } s\gg 0.\]}
 \end{theorem}
		\begin{proof}
			If \RED{$\vect(f)=(e_0,\dots, e_{\eta-1}) $, by \Cref{lemma:nilpotentwithconstants}
			\[ f = \sum_{i=0}^{\eta-1} [e_i]^{i} =  \sum_{i=0}^{\eta-1} \Sigma^i [e_i]. \]
If $\lead$ is the leading component, then $\nu_p(\lead)\leq \nu_p(e_i)$, $0\leq i<\gamma$, and 
$\nu_p(\lead)< \nu_p(e_i)$, $\gamma<i< \eta$.}
	Let us prove that  $ \tau(\Sigma^s f) = \tau(\Sigma^{s+\gamma} [\lead]) $ for $ s\gg 0 $.
Let \(\mu\) be the minimal natural number such that \(\eta - \gamma -1 < p^{\mu}(p-1).\)
			Notice that for any \(k \ge \mu\) both $ p^k $ and $ p^k +\eta - \gamma-1 $ are strictly less than $p^{k+1}$ and hence they have the same number of digits in base $ p $.
 In order to conclude the proof, we show that for any $k \ge \mu $ one has:
\[ \tau( f^{s}) = \tau ( [\lead]^{s+\gamma}) \qquad \forall  \, p^k- \gamma \le s < p^{k+1}- \gamma  \]
hence the statement holds for any $s\geq p^{\mu}- \gamma$.
	
	For \(s = p^{k}- \gamma \) we have:
	\begin{align*}
	f^{p^k - \gamma}  = [e_{0}]^{p^k - \gamma} + [e_{1}]^{1 + p^k - \gamma} + \dots + [e_{\gamma}]^{\gamma + p^k- \gamma} + \cdots + [e_{\eta-1}]^{\eta - 1 + p^k - \gamma}.
    \end{align*}
	By  Proposition~\ref{prop:periodconstants}, $ [e_{ \gamma}]^{p^k} $ has period $p^{\ell-\nu(\lead)+k} $.
The other summands have period strictly dividing \(p^{ \ell -\nu(\lead)+k}\):
\begin{itemize}
  \item For every $ \gamma < i < \eta $, \( p^k +i - \gamma < p^{k+1} \)  by construction and so $ p^k +i - \gamma $ has $ k+1 $ digits in base $ p $, hence the period of $[e_i]^{p^k +i - \gamma }   $ is $ p^{\ell-\nu_p(e_i)+k} \mid  p^{\ell-\nu_p(\lead)-1+k}
 $ (since $\nu_p(e_i)>\nu_p(\lead)$).
  \item For every $ 0 \le i < \gamma $, $ \nu_p(\lead) \le \nu_p(e_i) $ and $  p^k +i -\gamma < p^k $; hence \(p^{k}+i-\gamma\) has at most $ k $ digits in base $ p $. Thus the period of $ [e_i]^{p^k +i -\gamma }  $ is a divisor of $ p^{\ell-\nu_p(e_i) +k-1}$ and so it divides $p^{\ell-\nu_p(\lead )+k-1} $.
\end{itemize}
Thus the period \(\tau( f^{p^k-\gamma} )\) is equal to  \(p^{\ell -\nu_p(\lead) +k} \).

For $ p^k-\gamma < s < p^{k+1}-\gamma $,  the period of \( [e_\gamma]^{s+\gamma} \) is \(p^{\ell -\nu_p(\lead)+k}\), and by Lemma~\ref{lemma:primitiveperiod} $\tau(f^s) \ge \tau(f^{p^k-\gamma})= p^{\ell -\nu_p(\lead)+k} $.
Furthermore, since \(p^{k+1}+\eta-\gamma-1<p^{k+2}\) we have
\[p^k-\gamma< s\leq s+\eta-1<p^{k+1}-\gamma+\eta-1< p^{k+2}.\]
Then
\(\tau([e_i]^{s+i})\leq p^{\ell -(\nu_p(\lead)+1)+k+1} \) for \(\gamma< i\leq\eta-1\), and \(\tau( [e_i]^{s+i})\leq p^{\ell -\nu_p(\lead)+k} \) for \(0\leq i\leq {\RED{\gamma}}-1\).
Thus
\(\tau(f^s) \le p^{\ell -\nu_p(\lead)+k} \), and hence \(\tau( f^s) = p^{\ell -\nu_p(\lead)+k} \).
\end{proof}

\begin{corollary}\label{rem:infinito}
\RED{Denoted by $\lead$ the leading component of the generating vector of $ f \in \dnil{p^\ell} $, one has that for $t\gg 0$}
\[\tau\Big(\sum_{i=0}^t f^{i}\Big)=\tau\Big([e_\gamma]^{t+\gamma}\Big).\]
\end{corollary}

\begin{remark}\label{rem:limiteinferiore}
In the proof of \Cref{theorem:leading_term} we computed explicitly how big $s$ has to be for the statement to hold. Precisely we have $s\geq p^{\mu}-\gamma$ where $\mu$ is minimal with respect to
$\eta-\gamma-1<p^{\mu}(p-1)$. In particular if $\gamma=\eta-1$, the condition becomes $s\geq 0$.
\end{remark}

\begin{example}
 Let us consider the  sequence
$ V_2 = [2,1,2,0,0,1,0,0] \in \dnil{4}
$ from \Cref{example:vieru}.
It has nilpotency index 5 and \RED{$\vect(V_2)=(2,3,2,3,2)$. Then}
  \[ V_2 = [2] + \Sigma [3] + \Sigma^2 [2] + \Sigma^3 [3] + \Sigma^4 [2].\]
The leading component is  $ e_3=3 $ and by \Cref{theorem:leading_term}
\RED{one has $\tau(\Sigma^s V_{2}) = \tau(\Sigma^{s+3}[3])$ for  $s \gg 0$. By \Cref{rem:limiteinferiore} one easily checks that for any $s\geq 0$, if $ s + 3 = \lfloor a_k \cdots a_1 a_0 \rfloor_2$, $a_k>0$, } we have
$\tau(\Sigma^s V_2) =  \tau(\Sigma^{s+3}[3])=2^{2+k}.$
\end{example}

\begin{corollary}
  With the notation of the theorem above, if \(\ell = 1\), i.e. considering sequences on the finite field \( \Z_{p}\), the leading component of a nilpotent sequence is always \(e_{\eta-1}\) and it leads the period for any \(s \ge 0\).
  Hence \(\tau(\Sigma^s f) = \tau(\Sigma^{s+ \eta-1} [e_{\eta-1}])\) for any \(s \ge 0\).
\end{corollary}

In the following theorem we show that constant sequences and their primitives \RED{have a key role also studying} the primitives of idempotent sequences. For $z\in\mathbb Z$ and $0<\eta
\in\N$, denote by $0\leq \bar{z}<\eta$ the remainder in the division of $z$ by $\eta$. 

\begin{theorem}\label{idem_constants}
Consider $ f \in \didem{p^{\ell}} $ with idempotency index \( \eta \) and generating vector \( \vect(f)=(e_0,..., e_{\eta-1})$. 
%
%
For every $s \ge 1$,one has:
\begin{equation*}
\Sigma^{s}f = \Delta^{\overline{- s}}f - \sum_{j=0}^{s-1} \Sigma^{j}[e_{\overline{j-s}}].
\end{equation*}
This provides the explicit decomposition in idempotent and nilpotent part. Moreover if $\lead$ is the leading component of $\vect (f)$, one has
\[\tau(\Sigma^s f) =\lcm\left(\tau(f), \tau(\Sigma^{s-\eta+\gamma} [\lead])\right) \quad \forall s\gg 0.\]

\begin{proof}
We proceed by induction on $ s $.
\begin{itemize}
\item For $ s  = 1 $ one has 
$
\Sigma f = \Sigma (\Delta^\eta f) = \Delta^{\eta -1} f - [e_{\eta-1}]=\Delta^{\overline{-1}}f-[e_{\overline{-1}}]$, and hence the thesis.
\item Suppose that the statement is true for $1\leq s = t\eta+\bar s $, $t\geq 0$; let us prove it for $ s + 1 = t'\eta+\overline{s+1}$.
Notice that $(t',\overline{s+1})=(t+1,0)$ if $\bar s= \eta-1$ and $(t',\overline{s+1})=(t,\bar s+1)$ otherwise. By inductive hypothesis we have:
\begin{align*}
 f^{s+1}= \Sigma(f^{s}) = & \Sigma (\Delta^{\overline{-s}}f - \sum_{j=0}^{s-1} [e_{\overline{j-s}}]^{j}) \\
= & \Delta^{\overline{-s-1}}f - [e_{\overline{-s-1}}] - \sum_{j=1}^{s} [e_{\overline{j-1-s}}]^{j} \\
= & \Delta^{\overline{-(s+1)}}f - \sum_{j=0}^{s} [e_{\overline{j-(s+1)}}]^{j}.
\end{align*}
%
%
\end{itemize}
\noindent This proves the first part of the statement.

For the period, firstly we have $\tau\left(\Delta^{j} f \right)=\tau(f)$ for any $j\in \N$ (since $f\in \didem{p^\ell}$).
Now let us denote by $g$ the nilpotent sequence $\sum_{j=0}^{\eta-1}  [e_{j}]^{j}$. The nilpotency index of $g$ is $\eta_g:=\max\{j:e_j\not=0\}+1$; clearly $\eta_g\leq \eta$, but the leading component of $\vect(g)$ and $\vect(f)$ is the same: $[\lead]^\gamma$. Notice that for any $s \ge \eta$:
\[ \sum_{j=0}^{s-1} [e_{\overline{j-s}}]^{j}=\sum_{0\leq j<\overline{s}} [e_{\overline{j-s}}]^{j} + \sum_{j=\bar s}^{s-1} [e_{\overline{j-s}}]^{j}=\sum_{0\leq j<\overline{s}} [e_{\overline{j-s}}]^{j} +\sum_{i=0}^{t-1}  g^{i\eta+\bar s}. \]
By \Cref{rem:infinito}, for $s \gg 0$ one has:
\[ \tau \left( \sum_{j=0}^{s-1} [e_{\overline{j-s}}]^{j} \right) =
\tau \left( \sum_{i=0}^{t-1}  g^{i \eta + \bar s} \right)=\tau \left( [\lead]^{(t-1)\eta + \bar s + \gamma} \right) =\tau\left([\lead]^{s-\eta+\gamma}\right).\]
\end{proof}
\end{theorem}

\begin{remark}\label{rem:limiteinferioreidem}
In the proof of \Cref{idem_constants}, to make the statement relative to the period true, we did two assumptions on $s$. First $s$ has to be greater or equal than the idempotency index $\eta$, second $(t-1)\eta+\overline s=s-\eta$ has to be greater or equal than $p^{\mu}-\gamma$ where $\mu$ is minimal with respect to
$\eta-\gamma-1<p^{\mu}(p-1)$ to have he possibility to apply \Cref{theorem:leading_term} as observed in \Cref{rem:limiteinferiore}.
\end{remark}

\begin{example}
Consider the sequence $f = [1,3,0] \in \per{4}$. It is idempotent of index $\eta = 6$ and $\vect(f)=(1,2,3,1,0,1)$.
The leading constant is $e_5=1$. By  \Cref{idem_constants} for $ s =8 $, one has:
\[\Sigma^8 f=\Delta^{\overline{-8}}f-\sum_{j=0}^7\Sigma^j[e_{\overline{j-8}}]=\Delta^{4}f-\sum_{j=0}^7\Sigma^j[e_{\overline{j+4}}].\]
Indeed $\Sigma^8 f=[0,
 0,
 0,
 0,
 0,
 0,
 0,
 0,
 1,
 3,
 0,
 1,
 1,
 2,
 1,
 1,
 1,
 2,
 1,
 1,
 2,
 1,
 1,
 2,
 2,
 0,
 2,
 2,
 2,
 0,$\\
 $
 2,
 2,
 3,
 3,
 2,
 3,
 3,
 2,
 3,
 3,
 3,
 2,
 3,
 3,
 0,
 1,
 3,
 0]$, $\Delta^{4}f=[0,1,3]$, and 
\[\sum_{j=0}^7\Sigma^j[e_{\overline{j+4}}]
%
=[0,
1,
3,
0,
1,
3,
0,
1,
2,
1,
1,
2,
3,
3,
2,
3
].\]
Finally, by \Cref{rem:limiteinferioreidem} if $s\geq \eta=6$, and $s-\eta=s-6\geq 2^0-5=-4$, i.e., globally, $s\geq 6$, we have 
\begin{align*}
\tau(\Sigma^s f) = \mathrm{lcm} \left( \tau(f), \tau(\Sigma^{s-1}[ 1 ]) \right)	= 3 \cdot 2^{2+k}
\end{align*}
where $s-1=\lfloor 1a_{k-1}\cdots a_0\rfloor_2$. In particular for $s=8$ we have \[\tau(\Sigma^8 f)=\mathrm{lcm} \left( \tau(f), \tau(\Sigma^{7}[ 1 ]) \right)=\mathrm{lcm} \left( 3, 16) \right)=3\cdot 2^{2+2}=48.\]
\end{example}

%

\section{Recursive formula for binomials coefficients in \(\Z_{p^{\ell}}\)} \label{section:recursive_formula}
\blue{In this section, we focus on the $s$-th binomial function:
\begin{align*}
\bin_s:  \N  & \longrightarrow  \Z_{p^{\ell}}\\
 n &\longmapsto \binom{n}{s}.
 \end{align*}
As shown in previous section, this function coincides with the $s$-th primitives $\Sigma^{s}[1]$ of the constant sequence $[1] \in \per{p^{\ell}}.$
If the expression of $ s $ in base $ p $ is one of the following:
\begin{align*}
\lfloor b_k \cdots b_{k-m} & \underbracket[0.1ex]{(p-1) \cdots (p-1)}_\ell  b_{k-m-\ell-1} \cdots b_0 \rfloor_p \\
\lfloor b_k \cdots b_{k-m} & \underbracket[0.1ex]{\hspace{2mm} 0 \hspace{7mm} \cdots  \hspace{7mm} 0 \hspace{2mm} }_\ell  b_{k-m-\ell-1} \cdots b_0\rfloor_p \\
\lfloor b_k \cdots b_{k-m} & \underbracket[0.1ex]{(p-1)\; 0 \hspace{2mm} \cdots \hspace{2mm} 0 \hspace{2mm}}_{\ell} \,  b_{k-m-\ell-1} \cdots b_0\rfloor_p
\end{align*}
where $ k > \ell $ and  $ 0 \le m \le k -\ell -1$, we prove that it is possible to link the $s$-th binomial function $\bin_{s}$ to $\bin_{s'}$ where $s'$ is obtained from $s$ by removing one of the explicit coefficients in its $p$-base expression.
Of course, in the general case this formula can be combined with the usual ones.
Notice that when $p=2$ and $\ell=2$, this provides a complete recursive formula for the binomial function $\bin_{s}$.
}

Firstly we need some definitions.

\begin{definition}\label{def:RA}
	Given a sequence $ f \in S_m:=\Z_{m}^\N $, a prime $ q $ and an integer $ t \ge 1 $, we call
	the $ j $-th $ q^t $-subsequence of $ f $ the element $h_j\in  \Z_{m}^{q^{t}}$ defined as
	 \[ h_j = (f(j q^t),f(j q^t + 1), \dots, f((j+1)q^t -1)) \qquad j \in \N. \]
	We denote by $ \Double(f,q^t)\in S_m $ the sequence obtained repeating $ q $ times the ordered $ q^t $-subsequences of $ f $:
	\[ \Double(f,q^t) = (\underbracket[0.1ex]{h_0, \dots, h_0}_{q}, \underbracket[0.1ex]{h_1, \dots, h_1}_{q}, \dots ). \]
We denote by $ \Alt(f,q^t) \in S_m$ the sequence obtained alternating $ (q-1)q^t $ zeros and the ordered $ q^t $-subsequences of $ f $:
	\[ \Alt(f,q^t) = (\underbracket[0.1ex]{0, \dots, 0}_{(q-1)q^t},h_0, \underbracket[0.1ex]{0, \dots, 0}_{(q-1)q^t}, h_1, \dots). \]
\end{definition}

\begin{Proposition}\label{prop:RA}
For any $f\in S_m$, $t\geq 1$, and $n'=\lfloor a_{r} \dots a_{t} a_{t-1} \dots a_{0} \rfloor_{q}$ one has
\begin{align*}
 \Double(f, q^{t}) (n) & =	f(n') \quad \text{ if } n=  \lfloor a_{r} \dots a_{t} \, \alpha \, a_{t-1} \dots a_{0} \rfloor_{q}, \; \forall \,  0 \le \alpha < q.
 \\
\Alt(f, q^{t}) (n) & =
  \begin{cases}
	f(n') \quad \text{ if } n=  \lfloor a_{r} \dots a_{t} (q-1) a_{t-1} \dots a_{0} \rfloor_{q} \\
	0 \qquad \text{ otherwise. }
  \end{cases}\end{align*}
\end{Proposition}
\begin{proof}
By Definition~\ref{def:RA}, given $\xi\in\mathbb N$, $0\leq\alpha<q$, $0\leq i<q^t$, one has
\begin{align*}
  \Double(f, q^{t})(\xi q^{t+1}+\alpha q^t+i) &= f(\xi q^t+i) \\
  \Alt(f, q^{t}) (\xi q^{t+1}+\alpha q^t+i) &= \begin{cases}
     f(\xi q^t+i) & \text{if }\alpha=q-1, \\
     0 & \text{otherwise}. \end{cases}
\end{align*}
Translating in the $q$-adic representation, we get the claim.
\end{proof}

\begin{example}
  \begin{itemize}
	\item 	The set of  $ 2 $-subsequences of $ f = [0,1,2,3,4,5]  \in \per{7}$ is \[ \{  [0,1], [2,3], [4,5]\}. \]

	\item If $ h = [1,2,3,4,5,6,7,8] \in \per{11} $, then:
		\begin{align*}
			\Double(h,2^2) = & [1,2,3,4,1,2,3,4,5,6,7,8,5,6,7,8] \\
			\Alt(h,2^2) = & [0,0,0,0,1,2,3,4,0,0,0,0,5,6,7,8].
		\end{align*}
Moreover
		  \begin{align*}
			\Double(h, 2^{2})( 2^{3}+ 2^2+3)=8=h(2^2+3) \qquad & \Double(h, 2^{2})( 2^{3}+2)=7=h(2^2+2) \\
			\Alt(h,2^2)(2^{3}+ 2^2+3)=8=h(2^2+3) \qquad & \Alt(h,2^2)(2^{3}+2)=0.
		  \end{align*}
  \end{itemize}

\end{example}

\begin{remark} \label{rmk:positional}
  Observe the following facts:
  \begin{itemize}
	\item For any $ q^t $ both $ \Double $ and $ \Alt $ are linear operators: for any $c_1,c_2\in \Z_{m}$ and  $f_1,f_2\in S_m$, it is
	\begin{align*}
		\Double(c_1 f_1 + c_2 f_2, q^t)=& c_1\Double(f_1, q^t)+c_2 \Double(f_2, q^t) \\
		\Alt(c_1 f_1 +  c_2 f_2, q^t)= & c_1\Alt(f_1, q^t) + c_2\Alt(f_2,q^t).
	\end{align*}

	\item If $ f \in \per{m} $ has period $ \tau $ and $ q^t \mid \tau $, then both $ \Double(f,q^t) $ and $ \Alt(f,q^t) $ have period $ q\tau $.
  \end{itemize}
 
\end{remark}

\begin{definition}\label{def:PiZ}
	If $ f,g \in \per{p^\ell} $, we write:
	\begin{itemize}
		\item $  f \peq g  $ if for any $ n \ge 0 $, $f(n)=0$ if and only if $g(n)=0$, otherwise
		$ \nu_p (f(n))  = \nu_p (g(n))\in \{0, \cdots, \ell-1\}$.
		\item $ \Pi_i (f) := \# \{ f(x)  \mid 0 \le x < \tau(f), \; \nu_p (f(x)) = i \} $  the number of coefficients with $ p $-adic valuation $ i $, for every $ 0 \le i < \ell $.
		\item $ Z(f) := \# \{ f(x)  \mid 0 \le x < \tau(f), \; f(x) = 0 \}$ the number of zeros.
	\end{itemize}
\end{definition}

Let us consider now the primitives $ \Sigma^s [1] =  \bin_s  $ of the constant sequence $ [1] $ in $ \per{p^\ell} $.
Suppose that $ p^k \le s < p^{k+1} $. The next results allow to link in certain cases the quantities $ \Pi_i(\bin_s), Z(\bin_s) $ to the quantities $ \Pi_i(\bin_{s'}), Z(\bin_{s'}) $ for some $s'$ with $ p^{k-1} \le s' < p^k $.

\begin{lemma}\label{seq:11}
With the notation above, suppose that $ k>\ell$, $0\leq m\leq k-\ell-1$, and that the expression of $ s $ in base $ p $ is:
	\[ s = \lfloor b_k \cdots b_{k-m} \underbracket[0.1ex]{(p-1) \cdots (p-1)}_\ell b_{k-m-\ell-1} \cdots b_0 \rfloor_p.
	\]
	Denote by
	\begin{align*} s' :&=s-\big(b_kp^k+(b_{k-1}-b_k)p^{k-1}+\cdots+(p-1-b_{k-m})p^{k-m-1}\big)\\
	&= \lfloor b_k \cdots b_{k-m} \underbracket[0.1ex]{(p-1) \cdots (p-1)}_{\ell-1} b_{k-m-\ell-1} \cdots b_0 \rfloor_p .
	\end{align*}
	Then $\bin_s \peq \Alt(\bin_{s'}, p^{k-m-\ell}) $. In particular $ \Pi_i (\bin_s) = \Pi_i(\bin_{s'}) $ and $
		Z(\bin_s)=   Z(\bin_{s'}) + (p-1)  p^{k+\ell-1} $. 
\end{lemma}
\begin{proof}
	The sequence $ \bin_s $ has period $ p^{\ell +k} $ by Proposition~\ref{prop:periodconstants}. For any \(0 \le n < p^{\ell +k}\), let \( n = \lfloor a_{k+\ell -1} \dots a_0 \rfloor_{p}\) be its expression in base \(p\). The \(n\)-th coefficient of \(\bin_s\) is:
	\[\begin{pmatrix}
		a_{k+\ell-1} & \cdots & a_{k+1} & a_k & \cdots & a_{k-m} & a_{k-m-1} \cdots a_{k-m-\ell} & a_{k-m-\ell-1} &\cdots& a_0 \\
		&&&b_k & \cdots & b_{k-m} & \underbracket[0.1ex]{(p-1) \cdots (p-1)}_\ell & b_{k-m-\ell-1} &\cdots& b_0
	\end{pmatrix}.  \]
Let \(n'\) be obtained from \(n\) by removing the coefficient \(a_{k-m-\ell}\). The \(n'\)-th  coefficient of $ \bin_{s'} $ is:
\[  \begin{pmatrix}
	a_{k+\ell-1} & \cdots & a_{k+1} & a_k & \cdots & a_{k-m} & a_{k-m-1} \cdots a_{k-m-\ell+1} & a_{k-m-\ell-1} &\cdots& a_0 \\
	&&&b_k & \cdots & b_{k-m} & \underbracket[0.1ex]{(p-1) \cdots (p-1)}_{\ell-1} & b_{k-m-\ell-1} &\cdots& b_0
	\end{pmatrix}.  \]
  By Proposition~\ref{prop:RA}, to conclude that $ \bin_s \peq \Alt(\bin_{s'}, p^{k-m-\ell}) $, it is enough to show that
\( \nu_{p}( \bin_{s}(n)) = \nu_{p} (\bin_{s'}(n'))  \) if \(a_{k-m-\ell} = p-1\) and \(\bin_s(n) = 0\) otherwise.  To prove this, we use Kummer's Theorem studying the number of borrows in the subtractions $n - s$ and $n' - s'$ in base $p$:
\begin{itemize}
	\item If $ a_{k-m-\ell} = p-1 $: 
	\begin{itemize}
		\item If  $ a_{k-m-\ell} $ lends, the number of borrows in $ s $ is one more than the number of borrows in $ s' $. However in both binomials there are at least $ \ell $ borrows (given by the remaining $ (\ell-1) $ coefficients equal to $ p-1 $), hence both binomials are zero modulo $ p^\ell $.
		
		\item If $ a_{k-m-\ell} $ does not lend,
		the number of borrows is the same for $ s $ and $ s' $. 
	\end{itemize}

	\item If $ a_{k-m-\ell} < p-1 $: the binomial $\bin_{s}(n) = 0 $ since again there are at least $ \ell $ borrows.
\end{itemize}
From the considerations above, we conclude that $ \bin_s \peq \Alt(\bin_{s'}, p^{k-m-\ell}) $.
Then immediately follows
\begin{align*}
	\Pi_i(\bin_s) = \Pi_i(\bin_{s'})  \qquad
	Z(\bin_s) =   Z(\bin_{s'}) + (p-1)\tau({\bin_{s'}}) = Z(\bin_{s'}) + (p-1)  p^{k+\ell-1}.
\end{align*}

\end{proof}
\begin{remark}
With the notation above, it is possible to verify that the proof of the previous lemma holds also for the case $s=\lfloor \underbracket[0.1ex]{(p-1) \cdots (p-1)}_{\ell} b_{k-\ell} \cdots b_0 \rfloor_p$ (which corresponds to $m=-1$).
\end{remark}

\begin{lemma}\label{seq:000} 
	With the notation above, suppose that  $ k> \ell$, $0\leq m\leq k-\ell-1$ and that  the expression of $ s $ in base $ p $ is:
	\[ s = \lfloor b_k \cdots b_{k-m} \underbracket[0.1ex]{0 \cdots 0}_\ell b_{k-m-\ell-1} \cdots b_0\rfloor_p .\]
	Denote by
	\begin{align*} s' :&=s-\big(b_kp^k+(b_{k-1}-b_k)p^{k-1}+\cdots+(b_{k-m}-b_{k-m+1})p^{k-m}-b_{k-m}p^{k-m-1}\big)\\
	&= \lfloor b_k \cdots b_{k-m} \underbracket[0.1ex]{0 \cdots 0}_{\ell-1} b_{k-m-\ell-1} \cdots b_0\rfloor_p.
	\end{align*}
	Then $ \bin_s \peq \Double(\bin_{s'}, p^{k-m-1}) $. In particular,
		$ \Pi_i(\bin_s) = p \cdot  \Pi_i(\bin_{s'})  $ and $
		Z(\bin_s) =   p \cdot Z(\bin_{s'}). $
\end{lemma}
\begin{proof}
The sequence $ \bin_s $ has period $ p^{\ell +k} $ by Proposition~\ref{prop:periodconstants}.
  Similarly to the previous lemma, for \( 0 \le n < p^{\ell + k}\) with \( n = \lfloor a_{k + \ell -1} \dots a_{0} \rfloor_{p} \), the coefficient \( \bin_s(n)  = \binom ns\) is:
	\[  \begin{pmatrix}
		a_{k+\ell-1} & \cdots & a_{k+1} & a_k & \cdots & a_{k-m} & a_{k-m-1} \cdots a_{k-m-\ell} & a_{k-m-\ell-1} &\cdots& a_0 \\
		&&&b_k & \cdots & b_{k-m} & \underbracket[0.1ex]{0 \quad \cdots  \quad 0}_\ell & b_{k-m-\ell-1} &\cdots& b_0
	\end{pmatrix}.  \]
	Let \(n'\) be obtained from \(n\) by removing the coefficient \(a_{k-m-1}\), hence the \(n'\)-th  coefficient of $ \bin_{s'} $ is:
	\[  \begin{pmatrix}
		a_{k+\ell-1} & \cdots & a_{k+1} & a_k & \cdots & a_{k-m} & a_{k-m-2} \cdots a_{k-m-\ell} & a_{k-m-\ell-1} &\cdots& a_0 \\
		&&&b_k & \cdots & b_{k-m} & \underbracket[0.1ex]{0 \quad \cdots \quad 0}_{\ell-1} & b_{k-m-\ell-1} &\cdots& b_0
	\end{pmatrix}.  \]
By Proposition~\ref{prop:RA}, to conclude that $ \bin_s \peq \Double(\bin_{s'}, p^{k-m-1}) $, it is enough to show that,
 for any value of  \(a_{k-m-1}\),
$\bin_{s'}(n')=0$ whenever $\bin_{s}(n)=0$, otherwise
\( \nu_{p}( \bin_{s}(n)) = \nu_{p} (\bin_{s'}(n'))  \).
To prove this, we use Kummer's Theorem studying the number of borrows in the subtractions $n - s$ and $n' - s'$ in base $p$:
	\begin{itemize}
		\item If $ a_{k-m-1} $ lends, then $ a_{k-m-2}=\dots =a_{k-m-\ell}=0 $ and they all lend.
So in this case in both  $ s $ and $ s' $ there are at least $ \ell $ borrows (notice that  $ a_{k-m} $ lends in $s'$); so the binomials are both equal to zero.
		
		\item If $ a_{k-m-1} $ does not lend, then the number of borrows remains the same in both the binomials.
	\end{itemize}
Henceforth we can conclude that  $ \bin_s \peq \Double(\bin_{s'}, p^{k-m-1}) $, thus:
	\[ \Pi_i(\bin_s) = p \cdot  \Pi_i(\bin_{s'})  \qquad Z(\bin_s) =   p \cdot Z(\bin_{s'}). \]
\end{proof}

\begin{remark}
Observe that if $\ell = 1$, i.e. the base ring $\Z/p \Z$ is a field, \Cref{seq:000} (resp. \Cref{seq:11}) reduces to removing a coefficient equal to 0 (resp. equal to $p-1$) in the expression of $s$ in base $p$. This is just a consequence of Lucas's theorem on binomial coefficients modulo $p$.
\end{remark}

In order to present the last result of this section, we need some preliminary definitions.

\begin{definition}
Given 	$ s = \lfloor b_k \cdots b_{k-m} \underbracket[0.1ex]{(p-1)\; 0 \cdots 0}_{\ell} \, b_{k-m-\ell-1} \cdots b_0\rfloor_p \in \N$ with $k > \ell$ and $  0 \le m \le k -\ell -1$, we denote by $E_{s}$ the following subset of $\{ 0, \dots, p^{k+\ell}-1\}$:
\begin{align*}
E_s:=\Big\{n\in \mathbb N: 0\leq n<p^{k+\ell},\ & n=\lfloor a_{k+\ell-1}\dots a_0\rfloor_p \text{ such that: } \\
	a_{k-m-1}=p-1& \qquad  a_{k-m-2} \neq 0 \\
a_{k-m-i} =0\quad \forall \, 3\leq i \leq \ell & \qquad a_{k-m-\ell -1}< b_{k-m-\ell -1} \\
a_j \ge b_j \quad  \forall\, j \in \{0, \dots, & \, k-m-\ell -2\}
  \cup \{ k-m, \dots, k \} \Big\}.
	\end{align*}
We denote by $\chi_{E_{s}} \in \per{p^{\ell}}$ the sequence:
\[ \chi_{E_{s}} = [e_{0}, \dots, e_{p^{k + \ell}-1}] \quad \text{where }
e_{i} =
\begin{cases}
1 \text{ if } i \in E_{s} \\
0 \text{ otherwise. }
\end{cases}
\]
\end{definition}

The definition above makes sense also for $ m = -1 $: in such a way we include also the case $s = \lfloor \underbracket[0.1ex]{(p-1)\; 0 \cdots 0}_{\ell} \, b_{k-\ell} \cdots b_0\rfloor_p $. It is easy to check that
\[ |E_{s} | =  p^{\ell -1}
 \left(\prod_{j=k-m}^k (p-b_j)\right)(p-1)\; b_{k-m-\ell-1}
	\left(\prod_{i = 0}^{k - \ell -m -2}(p-b_i)\right)\]
and hence $E_{s} = \emptyset $ if $b_{k-m-\ell-1} = 0 $.

\begin{lemma}\label{seq:101}
	With the notation above, suppose that $ k> \ell, 0\leq m\leq k-\ell-1$ and that the expression of $ s $ in base $ p $ is:
	\[ s = \lfloor b_k \cdots b_{k-m} \underbracket[0.1ex]{(p-1)\; 0 \cdots 0}_{\ell} \, b_{k-m-\ell-1} \cdots b_0\rfloor_p .\]
	Denote by
	\begin{align*} s' :&=s-\big(b_kp^k+(b_{k-1}-b_k)p^{k-1}+\cdots+(p-1-b_{k-m})p^{k-m-1}-(p-1)p^{k-m-2}\big)\\
	&= \lfloor b_k \cdots b_{k-m} \underbracket[0.1ex]{(p-1)\; 0 \cdots 0}_{\ell-1} \,  b_{k-m-\ell-1} \cdots b_0\rfloor_p.
	\end{align*}
%
%
%
Then $\bin_{s} \equiv_{\nu} \Double(\bin_{s'}, p^{k-m-2}) + p^{\ell-1} \chi_{E_{s}}$ and thus
\begin{align*}
		\Pi_i(\bin_s) & = p\cdot \Pi_i(\bin_{s'}) \qquad 0 \le i \leq \ell -2 \\
		\Pi_{\ell -1}(\bin_s) & = p\cdot \Pi_{\ell -1}(\bin_{s'}) + | E_{s} | \\
		Z(\bin_s) & = p\cdot Z(\bin_{s'}) - | E_{s} |.
	\end{align*}
\end{lemma}
\begin{proof}
	The sequence $ \bin_s $ has period $ p^{\ell +k} $ by Proposition~\ref{prop:periodconstants}.
  Similarly to the previous lemmas, for \( 0 \le n < p^{\ell + k}\) with \( n = \lfloor a_{k + \ell -1} \dots a_{0} \rfloor_{p} \), the coefficient \( \bin_s(n)  = \binom ns\) is:
	\[  \begin{pmatrix}
		a_{k+\ell-1}  \cdots  a_{k+1} & a_k  \cdots  a_{k-m} &a_{k-m-1} & a_{k-m-2}  \cdots a_{k-m-\ell} & a_{k-m-\ell-1} \cdots a_0 \\
		&b_k  \cdots  b_{k-m} & p-1 & \underbracket[0.1ex]{ 0 \quad \cdots  \quad 0}_{\ell-1} & b_{k-m-\ell-1} \cdots b_0
	\end{pmatrix}.  \]
	Let \(n'\) be obtained from \(n\) by removing the coefficient \(a_{k-m-2}\), hence the \(n'\)-th  coefficient of $ \bin_{s'} $ is:
		\[  \begin{pmatrix}
		a_{k+\ell-1}  \cdots  a_{k+1} & a_k  \cdots  a_{k-m} & a_{k-m-1} & a_{k-m-3}  \cdots a_{k-m-\ell}  & a_{k-m-\ell-1} \cdots a_0 \\
		&b_k  \cdots  b_{k-m} & p-1 & \underbracket[0.1ex]{ 0 \quad \cdots  \quad 0}_{\ell-2} & b_{k-m-\ell-1} \cdots b_0
	\end{pmatrix}.  \]
Let us use Kummer's Theorem to study the number of borrows in the subtractions $n - s$ and $n' - s'$ in base $p$:
\begin{itemize}
\item	if $a_{k-m-\ell -1}$ does not lend, the two binomials have the same number of borrows.
\item 	if $a_{k-m-\ell -1}$ lends, we have the following  cases:
\begin{itemize}
\item	if $a_{k-m-2}=a_{k-m-3} = \cdots = a_{k-m-\ell}=0 $, then both binomials have at least $\ell$ borrows and hence they are zero. 
\item
If $a_{k-m-3} = \cdots = a_{k-m-\ell}=0 $  but
$a_{k-m-2}\neq 0$, there are at least $\ell$ borrows in $s'$. In this situation there are at least $\ell-1$ borrows in $s$ and they are precisely $\ell -1$ when $n \in E_{s}$.
\item In the remaining cases, there exists an index $k-m-\ell\leq i\leq {k-m -3}$
	such that $a_i\not=0$, thus $a_{k-m -2}$ does not lend, so the borrows in $s$ and $s'$ are the same. 
	\end{itemize}
	\end{itemize}
This proves the statement. 
\end{proof}

\begin{remark}\label{rem:Egamma}
With the notation above, it is possible to verify that the proof of the previous lemma holds also for the case
$s =  \lfloor \underbracket[0.1ex]{(p-1)\; 0 \cdots 0}_{\ell} \, b_{k-\ell} \cdots b_0\rfloor_p$ (which corresponds to $m=-1$).
Moreover, observe that \Cref{seq:101} generalises \Cref{seq:000} if $p = 2$: indeed the hypotheses of \Cref{seq:000} imply $b_{k-m - \ell -1} = 0$ in \Cref{seq:101} and hence $E_{s} = \emptyset$.
\end{remark}

\begin{remark}\label{rem:s'}
Let $s=\lfloor b_k\cdots b_0\rfloor$. The construction of $s'$ in \Cref{seq:11,seq:000,seq:101} does not depend on the $(k-m-\ell)$-tail $b_{k-m-\ell-1}...b_0$. Therefore if $s$ and $s+i$ differ only on their $(k-m-\ell)$-tails, then $(s+i)'=s'+i$.
\end{remark}

\section{The case of $ \Z_4 $ and Vieru's sequence} \label{section:Vieru_sequence}
Let us focus on $ \Z_4 $: with the notation of the previous section,  we are considering $ p = 2 $ and $ \ell=2 $. Notice that in this case \Cref{seq:11,seq:101} allow to reduce each binomial coefficient to a smaller one, permitting to link any primitive $ [1]^s $ with $2^k\leq s<2^{k+1}$ to a primitive $ [1]^{s'} $ with $2^{k-1}\leq s'<2^k$.

As an example we provide a recursive formula for the zeros $ Z(s):=Z(v^s) $
of the primitives $v^s:=\Sigma^sv$ of the sequence \[v = [2,1,2,0,0,1,0,0] \in \per{4}, \] 
when $ 2^k \le s < 2^{k+1} $ for $k \ge 5$.
The sequence $v$ is the nilpotent part of the sequence of Example~\ref{example:vieru} and
 the one motivating this article, as explained in the historical introduction.
The sequence $ Z(s)$ is clearly a sequence of natural numbers.

To state our formula, we need some technical results.
Firstly observe that since $2 \cdot 2 = 0 $ in $\Z_{4}$, if $2^{k} \le s, t < 2^{k +1}$, then
\[ 2 \chi_{E_{s} \triangle E_{t}} := 2(\chi_{E_{s}}+\chi_{E_{t}})(n)=\begin{cases}
      2& \text{if }n\in E_{s} \triangle E_{t}, \\
      0& \text{otherwise.}
\end{cases}
\]

Furthermore if $s = \lfloor 10 b_{k-2} \dots b_{0} \rfloor_{2}$, the quantity $|E_s|$ is linked with the number $\mathfrak z(s)$ of 0's in the binary expansion of $s$ in the following way:
\begin{align*}
|E_{s}| &= 2 \cdot b_{k-2}\cdot 2^{\mathfrak z(\lfloor b_{k-3}\cdots b_0\rfloor_2)}\\
&=b_{k-2}\cdot 2^{\mathfrak z(s)}=\begin{cases}
    2^{\mathfrak z(s)}  & \text{if  }b_{k-2}=1, \\
     0 & \text{otherwise}.
\end{cases}
\end{align*}

The coefficients $\Pi_0(f)$, $\Pi_1(f)$, $Z(f)$ introduced in \Cref{def:PiZ} represent the number of 1 or 3, the number of 2, and the number of 0 in $f$, respectively.

If
$s =  \lfloor 1 b_{k-1} \cdots  b_0\rfloor_2$  and $t = \lfloor 1  b'_{k-1} \cdots  b'_0\rfloor_2$,
denote by $(s\mid t)$ the bitwise OR of $s$ and $t$, i.e., the number whose 2-adic representation has 1 in each bit position for which the corresponding bit of either $s$ or $t$ is 1.

%
%

\begin{lemma}\label{lemma:wisebit} 
	Let $k\geq 5$ and $2^k+2^{k-2}\leq s<2^k+2^{k-1}-4$. 
	Set $\mathfrak d_k$ equal to the $(2^{k-2}-4)$-sequence 
	$\mathfrak d_k(s):=\Pi_1(2(\chi_{E_{s+1}\triangle E_{s+3}}))$.
Then 
	\[\mathfrak d_k(s)=2^{\mathfrak z(s+1)}+2^{\mathfrak z(s+3)}-2\times 2^{\mathfrak z((s+1\mid s+3))}\]
	and 
	\[\mathfrak d_5=(4,8,4,4)\quad\text{and} \quad \mathfrak d_{k+1}=(2\times \mathfrak d_k,4,2^{k-1}, 2^{k-2}, 2^{k-2},\mathfrak d_k) \ \forall k\geq 5.\]\end{lemma}

The proof of the following Lemma is postponed in \Cref{ap}.

\begin{remark}\label{rem:wisebit} 
Let us link the sequence $\mathfrak d_k$ to two well known integer sequences.
Fixed $k\geq 5$ 
the sequence $\mathfrak d_k$ coincides with
	\[(2^{k-a(4)}, 2^{k-a(5)}... , 2^{k-a(2^{k-2}-1)})\]
	where $a(2^t)=t+1$ and $a(2^t+i)=1+a(i)$ for $t\geq 0$ and $0<i<2^t$ (see A063787 in the OEIS, the online encyclopedia of integer sequences). Noticed that $a(2^{t_1}+\cdots +2^{t_h})=h+t_h$ for
	$t_1>\cdots >t_h\geq 0$,
one can directly prove that $\mathfrak d_k(2^k+2^{k-2}+2^{t_1}+\cdots 2^{t_h}-4)=2^{k-h-t_h}$.	
	
	Denoted by  $wt(n)$  the Hamming weight of $n$, i.e., the number of 1's in the binary expansion of $n$,
	we have, for $2^k+2^{k-2}\leq s<2^k+2^{k-1}-4$ 
		\[\mathfrak d_k(s)=2^{wt(2^k+2^{k-1}-4-s)+1}.\]
The recurrence relation for $\mathfrak d_k(s)$ permits to compute a recurrence relation for the Hamming weight. Denoted by $w_h$ the Hamming weight of the numbers $\lfloor 1\rfloor_2$, ..., $\lfloor 2^{h+1}-4\rfloor_2$, we have
\[w_2=(1,1,2,1),\quad w_{h+1}=	(w_h,h,h,h+1,1,w_h+1) \quad \forall h\geq 2\]
where $w_h+1$ is the sequence obtained by $w_h$ increasing by one each entrance.	
\end{remark}

\begin{MRF} \label{formula:recursive_Vieru}
For $k\geq 5$ and $2^k\leq s<2^{k+1}$, denote:
\begin{align*}
	(c_1, c_2, c_3,c_4) := &  \; 2^{k-5}(48,32,40,44) \\
	(c_1', c_2', c_3',c_4') := & \; 2^{k-5}(48,40,44,48)\\
	(c_1'', c_2'', c_3'',c_4'') := & \; 2^{k-5}(32,32,48,64)\\
	\mathcal Z_k:=& \; \big(Z(s)\big)_{2^k\leq s< 2^{k+1}}.
\end{align*}

The initial condition is
\begin{align*}
\mathcal Z_5=(32,48,64,&88,64,80,88,92,64,80,88,104,92,104,108,94, \\
	& 78,88,96,108,96,104,108,110,102,108,112,118,114,118,120,64).
\end{align*}
For $ k \ge 6 $, the $ 2^k $-tuple $ \mathcal Z_k$ coincides with
\[Z(s)=\begin{cases}
    2Z(s-2^{k-1})  & \text{if }2^k\leq s\leq 2^k+2^{k-2}-5\quad (\mathbf A) \\
    Z(s-2^{k-1}-2^{k-3})+c_i & \text{if }s=2^k+2^{k-2}-5+i,\ i=1,2,3,4\quad (\mathbf{B})\\
    2Z(s-2^{k-1})-\mathfrak d_k(s)& \text{if }2^k+2^{k-2}\leq s\leq 2^k+2^{k-1}-5\quad (\mathbf C)\\
    Z(s-2^{k-1}-2^{k-2})+c'_i  & \text{if }s=2^k+2^{k-1}-5+i,\ i=1,2,3,4\quad (\mathbf D)\\
    Z(s-2^k)+2^{k+1}& \text{if }2^k+2^{k-1}\leq s\leq 2^{k+1}-5\quad (\mathbf{E})\\
    Z(s-2^k)+c''_i& \text{if }s=2^{k+1}-5+i,\ i=1,2,3,4\quad (\mathbf F).
    \end{cases}\]

\end{MRF}

\textit{Sketch of the Proof}.
The $s$-primitive of the sequence $v= [2,1,2,0,0,1,0,0]$ is
equal to 
\begin{equation}
  \label{eq:summand}
  v^s = 2\bin_{s+4}+3\bin_{s+3}+2\bin_{s+2}+3\bin_{s+1}+2\bin_s \quad \forall s\geq 0.
\end{equation}
In base 2 we have 
\[ 2^k= \lfloor 10_k \rfloor_{2}:= \lfloor 1  	\underbracket[0.1ex]{0\cdots 0}_{k\text{ times}} \, \rfloor_{2},\] therefore
$\lfloor 10_k\rfloor_{2}   \leq [s]_2\leq \lfloor 11_k \rfloor_{2}$.
Set $h=k-5$, we will consider in order the following cases:
\[\mathbf A:\qquad \lfloor 1000_h000\rfloor_{2}\leq s \leq \lfloor 1001_h011\rfloor_{2};\]
\[\mathbf C:\qquad \lfloor 1010_h000\rfloor_{2}\leq  s \leq \lfloor 1011_h011\rfloor_{2};\]
\[\mathbf{E}:\qquad \lfloor 1100_h000\rfloor_{2}\leq  s \leq \lfloor 1111_h011\rfloor_{2};\]
\[\mathbf{B}:\qquad \lfloor 1001_h100\rfloor_{2} \leq  s \leq \lfloor 1001_h111\rfloor_{2};\]
\[\mathbf{D}:\qquad \lfloor 1011_h100\rfloor_{2} \leq  s \leq \lfloor 1011_h111\rfloor_{2};\]
\[\mathbf F:\qquad \lfloor 1111_h100\rfloor_{2} \leq  s \leq \lfloor 1111_h111\rfloor_{2}.\]
%
In the cases \textbf{A}, \textbf{C} and \textbf{E}, the primitive indices of all the summands in \Cref{eq:summand} have the same prefix: 10 in the first two cases, and 11 in the last.
This allows to apply in parallel the recursive lemmas of \Cref{section:recursive_formula}.
The remaining twelve cases require  \textit{ad hoc} analysis.
\appendix

\section{Proofs of \Cref{lemma:wisebit} and \Cref{formula:recursive_Vieru} }\label{ap}
\begin{proof}[Proof of Lemma~\ref{lemma:wisebit}]
Observe that, by Remark~\ref{rem:Egamma}
$$\mathfrak d_k(s)=|E_{s+1}|+|E_{s+3}|-2\times |E_{s+1}\cap E_{s+3}|=
2^{\mathfrak z(s+1)}+2^{\mathfrak z(s+3)}-2\times 2^{\mathfrak z((s+1\mid s+3))}$$ 
since
$ |E_{s+1}\cap E_{s+3}|=2^{\mathfrak z((s+1\mid s+3))}$
(see definition of $E_s$ in the proof of Lemma~\ref{seq:101} which, in our case, reduces
$a_j\geq b_j$).

	If $k=5$, then $s\in\{40,41,42,43\}$. It is easy to verify that
	\[\mathfrak d_5=(2^{\mathfrak z(41)}+2^{\mathfrak z(43)}-2^{\mathfrak z(41\mid 43)+1}, ..., 2^{\mathfrak z(44)}+2^{\mathfrak z(46)}-2^{\mathfrak z(44\mid 46)+1})=(4,8,4,4).\]
	Fixed $k$, the binary representation of the numbers $s$ between  $2^{k}+2^{k-2}$ and $2^{k}+2^{k-1}-4$ are of the following three types
	\begin{itemize}
		\item $\RomanNumeralCaps{1}_k$: $2^{k}+2^{k-2}\leq s<2^{k}+2^{k-2}+2^{k-3}-4$,
		\item $\RomanNumeralCaps{2}_k$: $2^{k}+2^{k-2}+2^{k-3}-4\leq s<2^{k}+2^{k-2}+2^{k-3}$,
		\item $\RomanNumeralCaps{3}_k$: $2^{k}+2^{k-2}+2^{k-3}\leq s< 2^{k}+2^{k-1}-4$.
	\end{itemize}
	
Given $s'\in \RomanNumeralCaps{3}_{k+1}$, hence $s'=2^{k+1}+2^{k-1}+2^{k-2}+t$ with $0\leq t< 2^{k-2}-4$. Set $s=2^{k}+2^{k-2}+t$, we
get
$$(\mathfrak z(s'+1),\;  \mathfrak z(s'+3),\;   \mathfrak z(s'+1\mid s'+3))=(\mathfrak z(s+1),\;  \mathfrak z(s+3),\;   \mathfrak z(s+1\mid s+3)).$$

	Given $s'\in \RomanNumeralCaps{1}_{k+1}$, hence $s'=2^{k+1}+2^{k-1}+t$ with $0\leq t< 2^{k-2}-4$. Set $s=2^{k}+2^{k-2}+t$, we
get
$$(\mathfrak z(s'+1),\;  \mathfrak z(s'+3),\;   \mathfrak z(s'+1\mid s'+3))=(1+\mathfrak z(s+1),\;  1+\mathfrak z(s+3),\;   1+\mathfrak z(s+1\mid s+3)).$$

	Finally the binary representation of $s'$ in the group $II_{k+1}$ is the following:
	\[s':\quad \lfloor 10101_{k+1-5}00 \rfloor_{2}  ,\ \lfloor 10101_{k+1-5}01 \rfloor_{2} , \   \lfloor 10101_{k+1-5}10 \rfloor_{2} , \ \lfloor 10101_{k+1-5}11 \rfloor_{2} .\]
	Therefore $(\mathfrak z(s'+1), \mathfrak z(s'+3), \mathfrak z(s'+1\mid s'+3))$ are 
	\begin{enumerate}
		\item $(3,2,2)$ for $ s' =\lfloor 10101_{k+1-5}00 \rfloor_{2}$,
		\item $(3,k-1,2)$ for $ s' =\lfloor10101_{k+1-5}01\rfloor_{2}$,
		\item $(2,k-2,1)$ for $ s' =\lfloor10101_{k+1-5}10\rfloor_{2}$,
		\item $(k-1,k-2,k-2)$ for $s' =\lfloor10101_{k+1-5}11\rfloor_{2}$.
	\end{enumerate}
	Thus we get the wanted claim.
\end{proof}

\begin{proof}	[Proof of \Cref{formula:recursive_Vieru} ]
Using a generic computer algebra system one can easily compute 
the sequence $\mathcal Z_5$, the initial condition for the recursive formula.  

Cases \textbf{A} and \textbf{C}. 
In both the cases $s$, $s+1$, $s+2$, $s+3$, and $s+4$ have a binary representation $\lfloor 10b_{k-2}...b_0\rfloor$ with the two most representative figures equal to 10. If $ f \in \per{4} $, we denote shortly
	\begin{align*}
		\Double f^s := \Double(f^s,2^{k-1}) ,\quad
		\Alt f^s := \Alt(f^s,2^{k-1}).  
	\end{align*}

Using Lemma~\ref{seq:101},  we lead back the study of $\bin_s$, ..., $\bin_{s+4}$ to the study of  $\bin_{s'}$, ..., $\bin_{s'+4}$ where $ s' = s-(2^k-2^{k-1})=s-2^{k-1} $. It is
\begin{align*}
	v^s & = 2 \bin_{s+4} +3 \bin_{s+3} +2 \bin_{s+2} +3 \bin_{s+1} +2 \bin_{s} \\
	& \peq  2 \left( \Double \bin_{s'+4} + 2\chi_{E_{s+4}} \right) +3 \left( \Double \bin_{s'+3} + 2\chi_{E_{s+3}} \right) + 2\left( \Double \bin_{s'+2} + 2\chi_{E_{s+2}} \right) + \\
	& + 3 \left( \Double \bin_{s'+1} + 2\chi_{E_{s+1}} \right)  + 2 \left( \Double \bin_{s'} + 2\chi_{E_{s}} \right)\\
	& \peq  \Double v^{s'} + 3 \cdot 2\chi_{E_{s+1}} + 3\cdot 2\chi_{E_{s+3}} \\
	& \peq  \Double v^{s'} + 2\chi_{E_{s+1}\triangle E_{s+3}}.
\end{align*}
In case \textbf{A} it is $E_{s+1}=\emptyset=E_{s+3}$ and hence $v^s\peq \Double v^{s'}$. Therefore 
\[Z(v^{s})=Z\left(\Double v^{s'}\right)=2\times Z(v^{s'}).\]
In case \textbf{C}, if $n\in E_{s+1}\triangle E_{s+3}$, then $\Double v^{s'}(n)$ is equal to zero.
Indeed it is easy to check that $n = \lfloor a_{k+1} \dots a_{0} \rfloor_{2} \in E_{s+1}\triangle E_{s+3}$ implies $a_{k} = 1$, $a_{k-1} = 1$ and $a_{k-2} = 0$.
Since the binary representation of $t\in\{s, s+1, s+2, s+3,s+4\}$ is $\lfloor 1 0 1 b_{k-3} \dots b_{0} \rfloor_{2}$, using Kummer's Theorem one has for $t'=t-2^{k-1}$:
\begin{align*}
\Double \bin_{t'}(n) = \bin_{t'}(n') = \binom{\lfloor a_{k+1} 1 0 a_{k-3} \dots a_{0} \rfloor_{2}}{ {\phantom{AAi,}}\lfloor 1 1 b_{k-3} \dots b_{0} \rfloor_{2}} = 0,
\end{align*}
hence $\Double v^{s'}(n)=0$.
Therefore, 
we have
\[Z(v^{s})=Z\left(\Double v^{s'}\right)-\Pi_1(2\chi_{E_{s+1}\triangle E_{s+3}})=2\times Z(v^{s'})-\mathfrak d_k(s).\]

\medskip

Case \textbf{E}.
The numbers $s$, $s+1$, $s+2$, $s+3$, and $s+4$ have a binary representation $\lfloor 11b_{k-2}...b_0\rfloor$ with the two most representative figures equal to 11. If $ f \in \per{4} $, we denote shortly
	\begin{align*}
		\Double f^s := \Double(f^s,2^{k-1}) ,\quad
		\Alt f^s := \Alt(f^s,2^{k-1}).  
	\end{align*}
Using Lemma~\ref{seq:11},  we lead back the study of $\bin_s$, ..., $\bin_{s+4}$ to the study of  $\bin_{s'}$, ..., $\bin_{s'+4}$ where $ s' = s-2^{k} $. Thanks to the linearity of $ \Alt $ we have:
\begin{align*}
	v^s & = 2 \bin_{s+4} +3 \bin_{s+3} +2 \bin_{s+2} +3 \bin_{s+1} +2 \bin_{s} \\
	& \peq  2\Alt \bin_{s'+4} +3 \Alt \bin_{s'+3} + 2\Alt \bin_{s'+2} + 3\Alt\bin_{s'+1} + 2\Alt \bin_{s'} \\
	& \peq  \Alt \left( 2 \bin_{s'+4} +3 \bin_{s'+3} +2 \bin_{s'+2} +	3 \bin_{s'+1} +2 \bin_{s'} \right) \\
	& \peq  \Alt v^{s'}.
\end{align*}

Therefore
\[Z(v^s)=Z\left(\Alt v^{s'} \right)=Z(v^{s'})+2^{k+1}.\]

Case \textbf{B} and \textbf{D}. 
The number $s$ has a binary representation $\lfloor 10b_{k-2}1_{h}1b_1b_0\rfloor$ with $b_0,b_1, b_{k-2}\in\{0,1\}$. If $ f \in \per{4} $, we denote shortly
	\begin{align*}
		\Double f^s := \Double(f^s,2^{k-4}) ,\quad
		\Alt f^s := \Alt(f^s,2^{k-4}).  
	\end{align*}
B. Using \Cref{seq:11} with $m=2$ and \Cref{seq:000} with $m=3$,  we lead back the study of $\bin_s$, ..., $\bin_{s+4}$ to the study of  $\bin_{s'}$, ..., $\bin_{s'+4}$ where $ s' = s-2^{k-1}-2^{k-3} $ in case \textbf{B}, and $ s' = s-2^{k-1}-2^{k-2} $  in case \textbf{D}.
\begin{itemize}
	\item If $(b_1b_0)=(00)$, then we have
\end{itemize}	
	\[ s+1 = \lfloor 10b_{k-2}1_{h}101\rfloor_{2},\ s+2 = \lfloor 10b_{k-2}1_{h}110 \rfloor_{2},\ s+3 = \lfloor 10b_{k-2}1_{h}111\rfloor_{2},\]
and $	s+4 = \lfloor 1b'_{k-1}b'_{k-2}0_{h}000\rfloor_{2}$ with $b'_{k-1}b'_{k-2}=01$ in case \textbf{B} and $b'_{k-1}b'_{k-2}=10$ in case \textbf{D}.
By \Cref{seq:11} with $m=2$ for $s+i$, $i=0,1,2,3$ and \Cref{seq:000} with $m=3$ for $s+4$ we have 
\[\bin_{s+i}\peq \Alt \bin_{s'+i},\ i=0,1,2,3,\text{ and }\bin_{s+4} \peq \Double \bin_{s'+4}.\]
Then
	\[v^s \peq 2\Double\bin_{s'+4} +3 \Alt \bin_{s'+3} + 2 \Alt \bin_{s'+2} + 3 \Alt \bin_{s'+1} + 2 \Alt \bin_{s'}. \]
Analysing the previous equation in blocks of length $ 2^{k-4} $, one obtains:
	\[Z(v^s) = Z(v^{s'})+ Z\left(2 \bin_{s'+4}\right).\]
Since $s'+4=\lfloor 1b'_{k-1}b'_{k-2}0_{h-1}000\rfloor_2$, applying $h$-times \Cref{seq:000}, we get
\begin{align*}
Z\left(2 \bin_{s'+4}\right)&=Z\left( \bin_{s'+4}\right)+\Pi_1\left(\bin_{s'+4}\right) \\
&=
\begin{cases}
  2^{h}\Big(Z\left( \bin_{20}\right)+\Pi_1\left( \bin_{20}\right)\Big) =48\cdot 2^{k-5}   & \text{in case \textbf{B} }, \\
  2^{h-1}\Big(Z\left( \bin_{24}\right)+\Pi_1\left( \bin_{24}\right)\Big)= 48\cdot 2^{k-5}    & \text{in case \textbf{D} }
\end{cases}
\end{align*}
Therefore  in both the cases \textbf{B} and \textbf{D} we have $Z(v^s)=Z(v^{s'})+2^{k-5}\times 48$.
\begin{itemize}
	\item If $(b_1b_0)=(01)$, then we have $Z(v^s)=\begin{cases}
 Z(v^{s'})+2^{k-5}\times 32     & \text{in case \textbf{B}}, \\
   Z(v^{s'})+2^{k-5}\times 40     & \text{in case \textbf{D}}.
\end{cases}$
\end{itemize}
\begin{itemize}
	\item If $(b_1b_0)=(10)$, then we have $Z(v^s)=\begin{cases}
 Z(v^{s'})+2^{k-5}\times 40     & \text{in case \textbf{B}}, \\
   Z(v^{s'})+2^{k-5}\times 44     & \text{in case \textbf{D}}.
\end{cases}$
\end{itemize}
\begin{itemize}
	\item If $(b_1b_0)=(11)$, then we have $Z(v^s)=\begin{cases}
 Z(v^{s'})+2^{k-5}\times 44     & \text{in case \textbf{B}}, \\
   Z(v^{s'})+2^{k-5}\times 48     & \text{in case \textbf{D}}.
\end{cases}$
\end{itemize}

\medskip
Case \textbf{F}. 
The number $s$ has a binary representation $\lfloor 1111_{h}1b_1b_0\rfloor$ with $b_0,b_1\in\{0,1\}$. If $ f \in \per{4} $, and $ 2^k\leq t<2^{k+1} $we denote shortly
	\begin{align*}
		\Double f^t := \Double(f^t,2^{k-2}) ,\quad
		\Alt f^t := \Alt(f^t,2^{k-1}).  
	\end{align*}
We lead back the study of $\bin_s$, ..., $\bin_{s+4}$ to the study of  $\bin_{s'}$, ..., $\bin_{s'+4}$ where $ s' = s-2^{k} $
\begin{itemize}
	\item If $(b_1b_0)=(00)$, then we have
\end{itemize}
\[ s+1 = \lfloor 1111_{h}101 \rfloor_{2},\ s+2 = \lfloor 1111_{h}110\rfloor_{2},\
	s+3 = \lfloor 1111_{h}111\rfloor_{2},\ s+4 = \lfloor 10000_{h}000\rfloor_{2}.\]
	For  $0 \leq i \leq 3$ the sequence $\bin_{s+i}$ has period $2^{k+2}$, while $\bin_{s+4}$ has period $2^{k+3}$. Nevertheless, the period of
	\[v^s = 2 \bin_{s+4} + 3 \bin_{s+3} + 2  \bin_{s+2} + 3 \bin_{s+1} + 2 \bin_{s}\]
	is $2^{k+2}$: indeed the sequence $ 2 \bin_{s+4} $ has period $2^{k+2}$ by \Cref{prop:periodconstants}. By \Cref{seq:11} with $m=-1$, \Cref{seq:000} with $m=1$, and \Cref{rem:s'} we have
	\[
	v^s= 2 \Double \bin_{s'+4} + 3  \Alt \bin_{s'+3} + 2 \Alt \bin_{s'+2} + 3 \Alt \bin_{s'+1} +2  \Alt\bin_{s'}
	\]
where $s'=s-2^k$. Then one gets
	\[Z(v^s)=Z(v^{s'})+Z(2  \bin_{s'+4}).\]
	Notice that $Z(2 \bin_{s'+4})=\frac 12\big(Z\left(\bin_{s'+4}\right)+\Pi_1\left(\bin_{s'+4}\right)\big)$. Indeed $2  \bin_{s'+4}$ has period equal to one half of the period of $  \bin_{s'+4}$ and the $0s$ of 	$2  \bin_{s'+4}$ correspond to the $0s$ and $2s$ of $ \bin_{s'+4}$.
Applying $h$-times \Cref{seq:000} with $m=1$, we get
\[Z\left(\bin_{s'+4}\right)+\Pi_1\left(\bin_{s'+4}\right) =2^{h}\big(Z(\bin_{32})+\Pi_1(\bin_{32})\big)=2^{k-5}\cdot 64.\]
Hence $Z(v^s)=Z(v^{s'})+2^{k-5}\cdot 32$.
\begin{itemize}
 \item If $(b_1b_0)=(01)$, we have 
\end{itemize}
 \[ s+1 =  \lfloor 1111_{h}110 \rfloor_{2},\
	s+2 \lfloor 1111_{h}111 \rfloor_{2},\ s+3 = \lfloor 10000_{h}000 \rfloor_{2},\ \ s+4 = \lfloor 10000_{h}001 \rfloor_{2}.\]	
	By \Cref{seq:11} with $m=-1$, \Cref{seq:000} with $m=1$, and \Cref{rem:s'} we have
	\[v^s = 2 \Double\bin_{s'+4}+3\Double\bin_{s'+3}+ 2 \Alt\bin_{s'+2}+ 3 \Alt \bin_{s'+1} + 2  \Alt \bin_{s'} .\]
	Observe that $3\Double\bin_{s'+3}$ has period $2^{k+3}$, while $2 \Double\bin_{s'+4}$, $\Alt \bin_{s'+i}$, $i=0,1,2$, have period $2^{k+2}$. We have that
	\[Z(v^s)=Z(v^{s'})+Z(2\bin_{s'+4} + 3\bin_{s'+3}).\]
Applying  $h$ times \Cref{seq:000} with $m=1$, we get
\[
Z(2\bin_{s'+4} + 3\bin_{s'+3})=Z(2R^h\bin_{33}+3R^h\bin_{32})=2^hZ(2\bin_{33}+3\bin_{32})
=2^{k-5}\cdot 32.
\]
Therefore 	$Z(v^s)=Z(v^{s'})+2^{k-5}\cdot 32$.
\begin{itemize}
	\item If $(b_1b_0)=(10)$, then we have $Z(v^s)=Z(v^{s'})+2^{k-5}\cdot 48$.
	\item If $(b_1b_0)=(11)$, then we have $Z(v^s)=Z(v^{s'})+2^{k-5}\cdot 64$.
\end{itemize}
%
%
%
%
%
%

	\end{proof}

\end{document}